\documentclass[10pt]{article}
\usepackage{amsfonts}
\usepackage{dsfont}
\usepackage{amsmath,mathrsfs}
\usepackage{amsthm}
\usepackage{amssymb}
\usepackage{amscd}
\usepackage{graphicx}
\usepackage{float}
\usepackage{subfigure}
\usepackage{epstopdf}
\usepackage{appendix}
\usepackage{booktabs}
\usepackage{units}
\usepackage{enumerate}
\usepackage{mathrsfs}
\usepackage{longtable}
\allowdisplaybreaks[1]
\newtheorem{definition}{Definition}[section]
\newtheorem{theorem}{Theorem}[section]

\newtheorem{lemma}{Lemma}[section]
\newtheorem{proposition}{Proposition}[section]
\newtheorem{remark}{Remark}[section]
\newtheorem{example}{Example}[section]
\newtheorem{assumption}{Assumption}[section]

\newtheorem*{problem(P1)}{Problem (P1)}
\newtheorem*{problem(P2)}{Problem (P2)}
\newtheorem*{problem(P1a)}{Problem (P1a)}
\newtheorem*{problem(P1b)}{Problem (P1b)}
\newtheorem*{problem(P2a)}{Problem (P2a)}

\begin{document}
\title{{Social Optima in Linear Quadratic Graphon Field Control: Analysis via Infinite Dimensional Approach}\thanks{De-xuan Xu, Zhun Gou, and Nan-jing Huang: The work of these authors was supported by the National Natural Science Foundation of China (12171339).}}
\author{{De-xuan Xu$^a$, Zhun Gou$^b$ and Nan-jing Huang$^a$\footnote{Corresponding author,  E-mail: nanjinghuang@hotmail.com; njhuang@scu.edu.cn} }\\
{\scriptsize\it a. Department of Mathematics, Sichuan University, Chengdu,
Sichuan 610064, P.R. China}\\
{\scriptsize\it b. College of Mathematics and Statistics, Chongqing Technology and Business University, Chongqing 400067, P.R. China}}
\date{}
\maketitle
\vspace*{-9mm}
\begin{center}
\begin{minipage}{5.8in}
{\bf Abstract.}
This paper is concerned with linear quadratic graphon field social control problem where the noises of individual agents are correlated. Compared with the well-studied mean field system, the graphon field system consists of a large number of agents coupled weakly via a weighted undirected graph where each node represents an individual agent. Another notable feature of this paper is that the dynamics of states of agents are driven by Brownian motions with a correlation matrix. The infinite dimensional approach is adopted to design the centralized and decentralized controls for our large population system. By graphon theory, we prove that the linear quadratic (LQ) social optimum control problem under the centralized information pattern is equivalent to an LQ optimal control problem concerned with a stochastic evolution equation, and the feedback-type optimal centralized control is obtained. Then, by designing an auxiliary infinite dimensional optimal control problem through agent number $N\rightarrow\infty$, a set of decentralized strategies are constructed, which are further shown to be asymptotically social optimal.
\\ \ \\
\end{minipage}
\end{center}
\section{Introduction}
\qquad Strategic decision problems of large population and complex systems arise naturally with various applications such as financial markets, social networks, exhaustible resource production, network security. Classical models considering mean field systems with homogenous interaction, dates back to works of Boltzmann, Vlasov, McKean and others (see, for example, \cite{Kolokoltsov,McKean,Sznitman} and the references therein). The noncooperative game problems based on mean field systems, named as mean field games (MFG), were introduced by the parallel works of Huang, Caines, and Malham\'{e} in \cite{Huang,Huang1} and Lasry and Lions in \cite{Lions}, and have been extensively studied in recent decades (see for instance \cite{Carmona,Carmona1,Carmona2,Carmona4,Bensoussan,Bensoussan3} and the references therein). Apart from MFG, cooperative multi-agent decision problems (social optima) in mean field models have drawn more research interests due to their theoretical implications and real application potentials, see e.g. \cite{Huang2,Wang1,Jianhui,Wang2,Feng}, and references therein.

To consider large systems of agents whose interactions are not necessarily symmetric and possibly heterogeneous, it is natural to introduce graphs to describe how each agent is connected to the others. The theory of graphons (see, e.g., \cite{Janson,Lovasz}) provides a framework for the study of very-large systems of agents on large graphs, and has been used in, for example, \cite{Gao5,Carmona3,Aurell,Parise,Caines,Gao,Amini,Dexuan} to study large population games with heterogeneous interactions, which is called graphon mean field games (GMFG) or graphon field games (GFG). In addition, there exist some works concerning optimal control problems or social optima problems on large-scale graphs. For example, Gao and Caines \cite{Gao1} obtained the approximate optimal control of complex network systems in deterministic environment, by applying graphon theory and the theory of infinite dimensional systems. Crescenzo et al. \cite{Crescenzo} considered a family of controlled Brownian diffusion processes depending on the whole collection of marginal probability laws, and studied the optimal control of graphon mean field systems with heterogeneous interactions. In contrast to the assumption of mutual independence of Brownian motions required in \cite{Crescenzo}, Dunyak and Caines proposed a novel graphon mean field system in \cite{Dunyak1}, which allowed the Brownian motions in the dynamics to be correlated. Based on the models in \cite{Dunyak1}, Dunyak and Caines \cite{Dunyak2,Dunyak3} utilised graphon theory and Q-noise to study the linear quadratic optimal control problems on large graph, and shown that the optimal control problems on graphon systems are the limit of the corresponding finite graph optimal control problems. Recently, some very interesting works \cite{Liang,Liang1} by Liang et al. studied social optimal problems on graphs, where each node on the network represents a population. They obtained the centralized optimal control and derived a set of asymptotically optimal distributed controls when the number of population at each node is large. It is noteworthy that, for the finite-player game problems considered in their works, the graphs are not large, which differ from all the aforementioned references.

To our best knowledge, there is no work to study large population social optima problems with large-graph structures in continuous time. The present paper is thus devoted to the study the finite large population linear quadratic social optima problems in continuous time, where the states of all agents on large-scale graphs are governed by the dynamic systems which are driven by Brownian motions with a correlation matrix. We would like to mention that it makes sense to consider correlated noises in large population systems. For instance, some literature studied the situation where agents' states are affected by one or several common noises, as mentioned in \cite{Wang2,Feng1,zhenwu}.


To date, most previous works on large population social control adopted the fixed-point method (see, e.g., \cite{Huang2,Li,Salhab1}) and direct method (see, e.g., \cite{Wang2,Liang,Wang3}). However, due to the presence of large graph structure and complex correlations of noises, the aforementioned two traditional methods do not apply here in a straightforward manner. In the present paper, by employing the infinite dimensional stochastic analysis approach, we obtain the centralized and decentralized controls of the social optima problems. In particular, we first transform the $N$ correlated Brownian motions into $Q$-Brownian motions, and show that the centralized control of $N$-agent social control problem is equivalent to the solution of a stochastic optimal control problem in infinite dimension by deriving decomposed form of solution of the corresponding operator-valued Riccati equation. Besides, we construct and solve an auxiliary infinite dimensional control problem when size of the graph go to infinity, and then design a set of decentralized strategies from the solution of the auxiliary control problem. Finally, we show the asymptotic optimality of the decentralized strategies.

The main contributions of this work can be summarized from three different viewpoints. Firstly, we obtain the centralized optimal control and its feedback representation for the social optima problem by infinite dimensional approach. Secondly, we give a set of decentralized asymptotically optimal strategies for the social optima problem in a privacy-preserving pattern. Finally, different from the traditional methods mentioned above, we adopt a new approach via infinite-dimensional stochastic analysis to solve the large population social optima problems.

The rest of the paper is organized as follows. In the next section we recall the background for graphons and some basic definitions in infinite dimensional stochastic calculus. After that Section 3 introduces the social optima problem with finite agents. Section 4 obtains the feedback type optimal centralized control of the social optima problem with finite agents. Section 5 is devoted to design the decentralized strategies and study the asymptotic optimality of the decentralized strategies. In Section 6, we investigate the finite rank graphon case, before we summarize the results in Section 7.


\hspace*{\fill}\\
\noindent
\emph{Notation}: $L^p[0,1]$ denotes the Lebesgue space over $[0,1]$ under the norm defined by $\|\phi\|_p=\left(\int_0^1|\phi(\alpha)|^p d\alpha\right)^{\nicefrac{1}{p}}$. The inner product in $L^2[0,1]$ is defined as follows: for $\phi,\varphi\in L^2[0,1]$, $\langle \phi,\varphi \rangle=\int_0^1\phi(\alpha)\varphi(\alpha)d\alpha$. The function $\mathds{1}\in L^2[0,1]$ is defined as follows: $\mathds{1}(\alpha)\triangleq 1$, for all $\alpha\in[0,1]$. Denote by $I$ the identity operator on Hilbert spaces.

\section{Preliminaries}

\subsection{Graphs, graphons and graphon operators}
\qquad A graph $G=(V,E)$ is denoted by a node set $V=\{1,\ldots,N\}$ and an edge set $E\subset V\times V$. The corresponding adjacency matrix is defined by $M_N=[m_{ij}]$, where $m_{ij}\in [-1,1]$ denotes the weight between nodes $i$ and $j$. A graph is undirected if its edges is unordered, which implies that the corresponding adjacency matrix is symmetric.

In this paper, a graphon is defined as a symmetric measurable function from $[0,1]^2$ to $[-1,1]$. This definition is also adopted in \cite{Dexuan,Gao1,Gao4}. The name of ``graphon'' comes from the contraction of graph-function \cite{Lovasz}. There are also some different definitions for the term ``graphon'' in early literatures, we refer the reader to \cite{Lovasz,Gao5,Borgs2}. Specially, for the $N$-uniform partition $\{P_1,\ldots,P_N\}$ of $[0,1]$: $P_l=[\frac{l-1}{N},\frac{l}{N})$ for $1\leq l\leq N-1$ and $P_N=[\frac{N-1}{N},1]$, the step function type graphon $M^{[N]}$ corresponding to $M_N$ is given by
\begin{equation}\label{stepgraphon}
  M^{[N]}(\alpha, \beta)=\sum\limits_{q=1}^N \sum\limits_{l=1}^N \mathds{1}_{P_q}(\alpha)\mathds{1}_{P_l}(\beta)m_{ql},\qquad \forall (\alpha,\beta)\in [0, 1]^2,
\end{equation}
where $\mathds{1}_{P_q}(\cdot)$ is the indicator function.

A graphon operator $T_M:L^2[0,1]\to L^2[0,1]$ is defined as follows
\begin{equation}\label{graphonoperator}
(T_M\varphi)(\alpha)=\int_{[0,1]}M(\alpha,\beta)\varphi(\beta)d\beta,\quad \forall\varphi\in L^2[0,1],
\end{equation}
where $M(\cdot,\cdot)$ is a graphon. To ease notation, we are inclined to denote graphon operators by $M$ instead of $T_M$ when no confusion occurs. The graphon operator $M$ is self-adjoint and compact \cite{Lovasz}.
\begin{lemma}\label{buchongGaoprop}
Consider a real symmetric matrix $M_N=[m_{ij}]$, $m_{ij}\in[-1,1]$ with nonzero eigenvalues $\{\lambda_l^{M_N}\}_{l=1}^r$ and orthonormal eigenvector $\{v_l^{M_N}\}_{l=1}^r$, then the integral operator $M^{[N]}$ defined by (\ref{stepgraphon})-(\ref{graphonoperator}) has $rank\ M^{[N]}=r$.
\end{lemma}
\begin{proof}
By \cite[Prop. 3]{Gao4}, $\{\sqrt{N}S_{v_l^{M_N}}\}_{l=1}^r$ are a set of orthonormal eigenfunctions of operator $M^{[N]}$, and $\{\frac{1}{N}\lambda_l^{M_N}\}_{l=1}^r$ are corresponding nonzero eigenvalues. Suppose $rank\ M^{[N]}>r$. Then we can find a eigenfunction $f\in L^2[0,1]$ corresponding to a nonzero eigenvalue, such that $\langle f,\sqrt{N}S_{v_l^{M_N}}\rangle=0$, $l=1,\cdots,r$, which implies $\sum_{i=1}^N v_l^{M_N}(i)\langle \mathds{1}_{P_i},f\rangle=0$, $l=1,\cdots,r$. Define $\vec{f}=(\langle \mathds{1}_{P_1},f\rangle,\cdots,\langle \mathds{1}_{P_N},f\rangle)^T$. Then we obtain $\langle v_l^{M_N},\vec{f}\rangle=0$, $l=1,\cdots,r$, and so $M_N\vec{f}=0$. Hence
\begin{align*}
  \int_0^1 M^{[N]}(\alpha,\beta)f(\beta)d\beta & =\int_0^1 \Big(\sum_{i=1}^N\sum_{j=1}^N \mathds{1}_{P_i}(\alpha)\mathds{1}_{P_j}(\beta)m_{ij}\Big)f(\beta)d\beta
  =\sum_{i=1}^N \mathds{1}_{P_i}(\alpha)\sum_{j=1}^Nm_{ij}\langle \mathds{1}_{P_j},f\rangle=0.
\end{align*}
Then the corresponding eigenvalue of $f$ is zero, which is a contradiction.
\end{proof}
Combining \cite[Prop. 3]{Gao4} and Lemma \ref{buchongGaoprop}, $\{\frac{1}{N}\lambda_l^{M_N}\}_{l=1}^r$ contains all nonzero eigenvalues of operator $M^{[N]}$.

\subsection{Basic definitions in infinite dimensional stochastic calculus}
\qquad We denote by $(H,\langle\cdot\rangle_H)$ a separable, real Hilbert space. In order to introduce the notation of Q-wiener process, we proceed as follows. Denote by $\mathscr{L}_1(H)$ the space of trace class operators on $H$,
$$
\mathscr{L}_1(H)=\{L\in\mathscr{L}(H):tr\big((LL^*)^{\frac{1}{2}}\big)<\infty\}
$$
where $\mathscr{L}(H)$ is the space of bounded linear operators on $H$, and the trace of the operator $(LL^*)^{\frac{1}{2}}$ is defined by
$$
tr\big((LL^*)^{\frac{1}{2}}\big)=\sum\limits_{j=1}^\infty\langle (LL^*)^{\frac{1}{2}}e_j,e_j\rangle_H
$$
for an orthonormal basis $\{e_j\}_{j=1}^\infty\subset H$. The trace $tr\big((LL^*)^{\frac{1}{2}}\big)$ is independent of the choice of the orthonormal basis.
\begin{definition}[$Q$-Wiener Process\cite{Gawarecki}]\label{Q-Wiener}
Let $(\Omega,\mathcal{F},\mathbb{P})$ be a complete probability space, $Q$ be a positive self-adjoint trace class operator on a separable Hilbert space $H$, $\{\varphi_i\}_{i=1}^\infty$ be an orthonormal basis in $H$ diagonalizing $Q$, and $\{\lambda_i\}_{i=1}^\infty$ be the corresponding eigenvalues, i.e. $Q\varphi_i=\lambda_i\varphi_i$ for $i=1,2,\cdots.$ Let $\{w_i\}_{i=1}^\infty$ be a sequence of real valued Brownian motions mutually independent on $(\Omega,\mathcal{F},\mathbb{P})$. The process
$$
W_t^Q=\sum\limits_{i=1}^\infty\sqrt{\lambda_i}w_i(t)\varphi_i
$$
is called a $Q$-Wiener process in $H$.
\end{definition}
We denote by $L_\mathcal{F}^2(0,T;H)$ the set of all $\{\mathcal{F}_t\}_{t\in[0,T]}$-progressively measurable process $X(\cdot)$ taking values in $H$ such that $E\int_0^T \|X(t)\|_H^2dt<\infty$.
Denote by $L_\mathcal{F}^2(\Omega;C([0,T];H))$ the set of all $\{\mathcal{F}_t\}_{t\in[0,T]}$-progressively measurable continuous process $X(\cdot)$ taking values in $H$, such that
$
E\sup\limits_{t\in [0,T]}\|X(t)\|_H^2<\infty.
$
Denote by $C_\mathcal{F}([0,T];L^2(\Omega,H))$ the set of all $\{\mathcal{F}_t\}_{t\in[0,T]}$-progressively measurable process $X(\cdot)$ taking values in $H$, such that
$X(\cdot):[0,T]\rightarrow L_{\mathcal{F}_T}^2(\Omega;H)$ is continuous.

It is noted that the paper involves the eigenvalues and eigenvectors of several operators and matrices. To facilitate the reader's understanding, a table summarizing the symbols used for all eigenvalues and eigenvectors appearing in the sequel is provided below.

\begin{longtable}{l l}
\caption{Notation of eigenvalues and eigenvectors} \\
\hline
$\lambda_l^{M_N}$ & an eigenvalue of adjacency matrix $M_N$ \\
$v_l^{M_N}$ & an eigenvector of adjacency matrix $M_N$ \\
$\lambda_l^{Q_N}$ & an eigenvalue of correlation matrix $Q_N$ \\
$v_l^{Q_N}$ & an eigenvector of correlation matrix $Q_N$ \\
$\lambda_l^M$ & an eigenvalue of graphon operator $M$ \\
$f_l^M$ & an eigenfunction of graphon operator $M$ \\
$\lambda_l^{M^{[N]}}$ & an eigenvalue of graphon operator $M^{[N]}$ \\
$f_l^{M^{[N]}}$ & an eigenfunction of graphon operator $M^{[N]}$ \\
$\lambda_l^Q$ & an eigenvalue of trace class operator $Q$ \\
$f_l^Q$ & an eigenfunction of trace class operator $Q$ \\
$\lambda_l^{Q^{[N]}}$ & an eigenvalue of trace class operator $Q^{[N]}$ \\
$f_l^{Q^{[N]}}$ & an eigenfunction of trace class operator $Q^{[N]}$ \\
\hline
\end{longtable}

\section{Problem Formulation}
\qquad Supposed that $(\Omega,\mathcal{F},\mathbb{P})$ is a complete probability space on which $\{\widetilde{W}_t^i, 0\leq t\leq T\}_{i=1}^N$ are a set of one-dimensional standard Brownian motions with a nonnegative definite correlation matrix $Q_N=[\rho_{ij}]$, i.e. $d\widetilde{W}_t^i d\widetilde{W}_t^j=\rho_{ij}dt$.

We consider a weakly coupled large population system with $N$ individual agents, which distributed over a weighted undirected graph with $N$-node represented by its adjacency matrix $M_N=[m_{ij}]$. Each node here represents an agent. The $i$th agent evolves as the following dynamics
\begin{equation}\label{x_i}
  dx_t^i=\big[Ax_t^i+Bu_t^i+b\frac{1}{N}\sum\limits_{j=1}^N m_{ij}x_t^j\big]dt+\sigma d\widetilde{W}_t^i,
\end{equation}
where $A,B,b,\sigma$ are given constants, $x_t^i, u_t^i\in \mathbb{R}$ are the state and input of the $i$th agent, respectively, $\frac{1}{N}\sum_{j=1}^N m_{ij}x_t^j$ is the network state average, $\{x_0^i\}_{i=1}^N$ are initial conditions and are assumed to be deterministic for convenience. To the best knowledge of the authors, this kind of state models are first proposed in \cite{Dunyak1}.

Denote $\mathbf{u}=\{u^1,\ldots,u^N\}$. The cost function of the $i$th agent is defined by
\begin{equation}\label{ithcost}
  J^i(\mathbf{u})=\frac{1}{2}E\int_0^T\Big[Q\big(x_t^i-\Gamma\frac{1}{N}\sum\limits_{j=1}^N m_{ij}x_t^j\big)^2+R(u_t^i)^2\Big]dt+\frac{1}{2}EQ_T|x_T^i|^2
\end{equation}
where $Q,Q_T\geq 0$, $R>0$, $\Gamma\in \mathbb{R}$. The social cost of $N$ agents is
\begin{equation}\label{socialcost}
  J_{soc}^N(\mathbf{u})=\sum\limits_{j=1}^NJ^i(\mathbf{u}).
\end{equation}

We note that $\{\widetilde{W}^i\}_{i=1}^N$ can be transformed into a set of independent Brownian motions. The first case is that matrix $Q_N$ is positive definite. Thus, $Q_N=CC^T$ with  $$C=[c_{ij}^N]=\big(v_1^{Q_N},\cdots,v_N^{Q_N}\big)diag\Big\{\sqrt{\lambda_1^{Q_N}},\cdots,\sqrt{\lambda_N^{Q_N}}\Big\},$$ where
$\{v_i^{Q_N}\}_{i=1}^N$ and $\{\lambda_i^{Q_N}\}_{i=1}^N$ are eigenvectors and eigenvalues of $Q_N$.  It follows that
\begin{equation}\label{3a}
  C^{-1}\big(d\widetilde{W}_t^1, \cdots, d\widetilde{W}_t^N\big)^T\big(d\widetilde{W}_t^1, \cdots, d\widetilde{W}_t^N\big)\big(C^T\big)^{-1}=I_Ndt.
\end{equation}
Let $(W_t^1,\cdots, W_t^N)^T\triangleq C^{-1}(\widetilde{W}_t^1,\cdots, \widetilde{W}_t^N)^T$. From (\ref{3a})
and L\'{e}vy's characterization of Brownian motion, one has that $\{W^i\}_{i=1}^N$ are $N$ independent Brownian motions. The second case is that matrix $Q_N$ is nonnegative definite, but not positive definite. Denote $rank(Q_N)=d^N$ with $d^N<N$. Let $C_1=\big(v_1^{Q_N},\cdots,v_{d^N}^{Q_N}\big)diag\Big\{\sqrt{\lambda_1^{Q_N}},\cdots,\sqrt{\lambda_{d^N}^{Q_N}}\Big\}$, which is full column rank. Then one may find another full column rank matrix $C_2$ such that $C\triangleq(C_1,C_2)$ is invertible, which yields
\begin{equation*}
  Q_N=C_1C_1^T=\left(\begin{array}{cc}
                C_1,& C_2
              \end{array}
              \right)
\begin{pmatrix}
I_{d^N} &    \\
 & 0
\end{pmatrix}
\left(\begin{array}{c}
                C_1^T\\
                C_2^T
              \end{array}
              \right)
=C
\begin{pmatrix}
I_{d^N} &    \\
 & 0
\end{pmatrix}
C^T.
\end{equation*}
Then
\begin{equation*}
  C^{-1}\big(d\widetilde{W}_t^1, \cdots, d\widetilde{W}_t^N\big)^T\big(d\widetilde{W}_t^1, \cdots, d\widetilde{W}_t^N\big)\big(C^T\big)^{-1}=\begin{pmatrix}
I_{d^N} &    \\
 & 0
\end{pmatrix}
  dt.
\end{equation*}
Let $(W_t^1,\cdots, W_t^N)^T\triangleq C^{-1}(\widetilde{W}_t^1,\cdots, \widetilde{W}_t^N)^T$, which yields
\begin{equation}\label{3d}
  \big(dW_t^1,\cdots,dW_t^N\big)^T\big(dW_t^1,\cdots,dW_t^N\big)=
  \begin{pmatrix}
I_{d^N} &    \\
 & 0
\end{pmatrix}
  dt.
\end{equation}
By (\ref{3d}) and L\'{e}vy's characterization of Brownian motion, one has that $\{W^i\}_{i=1}^{d^N}$ are $d^N$ independent Brownian motions. Moreover, the quadratic variations of $\{W^i\}_{i=d^N+1}^N$ are zero. By Proposition 1.12 in \cite{Revuz}, one has $W_t^i=W_0^i=0$, a.s. for every $t$, $d^N+1\leq i\leq N$. Then we derive
\begin{equation*}
  \big(W_t^1,\cdots,W_t^{d^N},0,\cdots,0\big)^T=C^{-1}\big(\widetilde{W}_t^1,\cdots,\widetilde{W}_t^N\big)^T,\ a.s.,\ \forall t\in[0,T],
\end{equation*}
and
\begin{equation}\label{nonnegativedefinite}
  (\widetilde{W}_t^1,\cdots,\widetilde{W}_t^N)^T=C\big(W_t^1,\cdots,W_t^{d^N},0,\cdots,0\big)^T=C_1\big(W_t^1,\cdots,W_t^{d^N}\big)^T,\ a.s.,\ \forall t\in[0,T].
\end{equation}

Let $rank(Q_N)=d_N$, $1\leq d_N\leq N$. Then, based on the above discussion, the individual dynamics \eqref{x_i} can be rewritten as follows
\begin{equation*}
  dx_t^i=\big[Ax_t^i+Bu_t^i+b\frac{1}{N}\sum\limits_{j=1}^N m_{ij}x_t^j\big]dt+\sigma c_{i1}^N dW_t^1+\cdots+\sigma c_{id_N}^N dW_t^{d_N}.
\end{equation*}

In this paper, we apply the infinite dimensional approach to sequentially study the mean field social control problem under centralized and decentralized information patterns.
Let $\mathcal{F}_t=\sigma (\widetilde{W}_s^i, 0\leq s\leq t, 1\leq i\leq N)\vee\mathcal{N}$, where $\mathcal{N}$ denotes the set of all $\mathbb{P}$-null sets. Define the centralized control set as
$$
\mathcal{U}^c=\Big\{(u^1,\ldots,u^N)|u^i \mbox{ is } \mathcal{F}_t\mbox{-progressively measurable with}\ E\int_0^T|u_t^i|^2dt<\infty\Big\}.
$$
We first propose the following optimization problem.
\begin{problem(P1)}
Seek a set of centralized controls to optimize the social cost $J_{soc}^N$ for the system (\ref{x_i})-(\ref{socialcost}), i.e. $\inf_{\mathbf{u}\in\mathcal{U}^c}J_{soc}^N(\mathbf{u})$.
\end{problem(P1)}

In the centralized information pattern, it is necessary to assume that each individual agent knows the relationships and interactions among all agents (the adjacency matrix $M_N$ and the correlation matrix $Q_N$).
In practice, however, agents usually express a strong preference for their privacy to be protected in some way. Specifically, agents refrain from sharing with others the information for their social relationships. Therefore, in this scenario, it is difficult to know $M_N$ and $Q_N$ for each agent. In the privacy-preserving pattern, we assume that the agents are informed of the limits of $M_N$ and $Q_N$ in some sense (if the limits exist), which will be discussed in Section 5 to obtain a set of asymptotically optimal decentralized strategies.

For example, we can consider a social optima problem in the privacy-preserving pattern as follows. For a given community, the government can ascertain the interconnections (denoted as $M_N$ and $Q_N$) among the residents through certain methods, such as big data analysis. Due to privacy protection concerns, the government is restricted from disclosing these information to the residents. While the government can enable residents to achieve asymptotically optimal strategies by appropriately conveying the information of limits of $M_N$ and $Q_N$ to them.

It is well known that the decentralized strategies in traditional sense are adapted to the filtrations generated by their own state processes. Different from the former, the decentralized strategies discussed in this paper means that the strategies are no longer depend on $M_N$ and $Q_N$.

We are now in the position to define the decentralized control set as
$$
\mathcal{U}_{sd}=\Big\{(u^1,\ldots,u^N)|u^i\mbox{ is } \mathcal{G}_t^i\mbox{-progressively measurable such that } E\int_0^T|u_t^i|^2dt<\infty\Big\},
$$
where $\mathcal{G}_t^i=\sigma(x_s^i, W_s^k, 0\leq s\leq t, 1\leq k\leq d)\vee\mathcal{N}$ with $d\leq d_N$. Thus, we can consider the following asymptotic social optimality problem.


\begin{problem(P2)}
Seek s set of decentralized control laws $\hat{\mathbf{u}}=(\hat{u}^1,\cdots,\hat{u}^N)$ in $\mathcal{U}_{sd}$ to asymptotically optimize the social cost for system (\ref{x_i})-(\ref{socialcost}), i.e.
$$
\Big|\frac{1}{N}J_{soc}^N(\hat{\mathbf{u}})-\frac{1}{N}\inf_{\mathbf{u}\in\mathcal{U}_c}J_{soc}^N(\mathbf{u})\Big|=o(1).
$$
\end{problem(P2)}

\section{Optimal centralized control}
\qquad In this section, we first transform Problem (P1) into an infinite dimensional stochastic optimal control problem, and then obtain the centralized close-loop controls.

For any given $u^i\in \mathcal{U}_c$, (\ref{x_i}) has a unique solution. Indeed, let $X=(x^1,\cdots,x^N)^T$, $U=(u^1,\cdots,u^N)^T$, and $\widetilde{W}=(\widetilde{W}^1,\cdots,\widetilde{W}^N)^T$. Then we may rewrite the state equation (\ref{x_i}) as the following vector-valued SDE
\begin{equation}\label{vectorSDE}
  dX_t=\Big[AX_t+BU_t+b\frac{1}{N}M_N X_t\Big]dt+\sigma d\widetilde{W}_t.
\end{equation}
Clearly, \eqref{vectorSDE} has a unique solution $X\in L_{\mathcal{F}}^2(\Omega;C([0,T];\mathbb{R}^N))$ by the well known result for SDE.

Let $x_t^{[N]}=\sum_{i=1}^N x_t^i \mathds{1}_{P_i}$ and $u_t^{[N]}=\sum_{i=1}^N u_t^i \mathds{1}_{P_i}$. Then  $x_t^{[N]},\ u_t^{[N]}\in L^2[0,1]$ for all $t,\ \omega$. Since $x_t^{[N]}(\alpha)=x_t^i$ ($u_t^{[N]}(\alpha)=u_t^i$), $\alpha\in P_i$, there is a one-to-one mapping relationship between $X_t$ and $x_t^{[N]}$ ($U_t$ and $u_t^{[N]}$).
To derive the dynamic of $x_t^{[N]}$, we first show that $W_t^{[N]}\triangleq \sum_{i=1}^N\widetilde{W}_t^i \mathds{1}_{P_i}$ is a Hilbert space valued Wiener process. Indeed, let $Q^{[N]}(\alpha,\beta)=\sum_{i=1}^N\sum_{j=1}^N\mathds{1}_{P_i}(\alpha)\mathds{1}_{P_j}(\beta)\rho_{ij}$, one can define integral operator $Q^{[N]}$ as follows
$$
(Q^{[N]}\varphi)(\alpha)=\int_{[0,1]}Q^{[N]}(\alpha,\beta)\varphi(\beta)d\beta,\quad \forall\varphi\in L^2[0,1].
$$
Denote $S_{v_j^{Q_N}}=\sum_{i=1}^N v_j^{Q_N}(i) \mathds{1}_{P_i}$ for $j=1,\cdots,d_N$, where $v_j(i)$ is the $i$th elements of $v_j$. By \cite[Prop. 3]{Gao4} and Lemma \ref{buchongGaoprop}, $\{\frac{1}{N}\lambda_j^{Q_N}\}_{j=1}^{d_N}$ and $\{\sqrt{N}S_{v_j^{Q_N}}\}_{j=1}^{d_N}$ are nonzero eigenvalues and orthonormal eigenfunctions of operator $Q^{[N]}$, respectively.
Then
\begin{align}\label{3e}
  W_t^{[N]} &=\sum_{i=1}^N \mathds{1}_{P_i} \widetilde{W}_t^i=\sum_{i=1}^N \mathds{1}_{P_i} \sum_{j=1}^{d_N} c_{ij}^N W_t^j\nonumber\\
   &=\sum_{j=1}^{d_N} \sum_{i=1}^N \mathds{1}_{P_i} c_{ij}^N W_t^j=\sum_{j=1}^{d_N} \sum_{i=1}^N \mathds{1}_{P_i} \sqrt{\lambda_j^{Q_N}} v_j^{Q_N}(i) W_t^j=\sum_{j=1}^{d_N} \sqrt{\frac{\lambda_j^{Q_N}}{N}}\sqrt{N}S_{v_j^{Q_N}}W_t^j.
\end{align}
Since operator $Q^{[N]}$ is self-adjoint, positive and
$$
tr(Q^{[N]})=\sum_{j=1}^{d_N} \langle Q^{[N]}\sqrt{N}S_{v_j^{Q_N}},\sqrt{N}S_{v_j^{Q_N}}\rangle=\sum_{j=1}^{d_N}\frac{\lambda_j^{Q_N}}{N}<\infty,
$$
by (\ref{3e}), we derive that $W_t^{[N]}$ is a $Q^{[N]}$-Wiener process.

On the other hand, by applying properties of Bochner integral, we obtain, for a.s. $\omega$ and each $t$,
$$
\sum_{i=1}^N x_t^i \mathds{1}_{P_i}=\sum_{i=1}^N x_0^i \mathds{1}_{P_i}+\int_0^t \Big[A\sum_{i=1}^N x_s^i \mathds{1}_{P_i}+B\sum_{i=1}^N u_s^i \mathds{1}_{P_i}+b\sum_{i=1}^N \Big(\frac{1}{N}\sum_{j=1}^N m_{ij} x_s^j\Big)\mathds{1}_{P_i} \Big]ds+\sigma\sum_{i=1}^N \widetilde{W}_t^i \mathds{1}_{P_i}.
$$
Then $x_t^{[N]}$ satisfies,  for arbitrary $t$, a.s. $\omega$,
\begin{equation}\label{xt[N]}
  x_t^{[N]}=x_0^{[N]}+\int_0^t (Ax_s^{[N]}+Bu_s^{[N]}+bM^{[N]}x_s^{[N]})ds+\sigma W_t^{[N]},
\end{equation}
where operator $M^{[N]}$ is defined by (\ref{stepgraphon}) and (\ref{graphonoperator}). For $u^i\in\mathcal{U}_c$, it is easy to verify $u^{[N]}\in L_\mathcal{F}^2([0,T];L^2[0,1])$, and by applying the Dominated Convergence Theorem (DCT), we can deduce $x^{[N]}\in C_\mathcal{F}([0,T];L^2(\Omega,L^2[0,1]))$. Moreover, since $\int_0^T|x_t^i|^2dt$ and $\int_0^T|u_t^i|^2dt<\infty$, a.s., one has
\begin{align*}\label{unknown5}
  & \int_0^T\|(AI+bM^{[N]})x_t^{[N]}+Bu_t^{[N]}\|dt\nonumber\\
  \leq & \sqrt{T}\|AI+bM^{[N]}\|_{op}\Big(\frac{1}{N}\sum_{i=1}^N\int_0^T|x_t^i|^2dt\Big)^\frac{1}{2}+|B|\sqrt{T}\Big(\frac{1}{N}\sum_{i=1}^N\int_0^T|u_t^i|^2dt\Big)^\frac{1}{2}<\infty\quad a.s..
\end{align*}
Thus, it follows from \eqref{xt[N]} that $L^2[0,1]$-valued process $x_t^{[N]}=\sum_{i=1}^N x_t^i \mathds{1}_{P_i}$ is a strong solution of the following stochastic evolution equation (SEE):
\begin{equation}\label{yt}
  \left\{\begin{split}
            &dy_t=\big[(AI+bM^{[N]})y_t+Bu_t^{[N]}\big]dt+\sigma dW_t^{[N]},  \\
             &y_0=x_0^{[N]}=\sum_{i=1}^N x_0^i \mathds{1}_{P_i}.
         \end{split}
  \right.
\end{equation}
In addition, for given $x_0^{[N]}=\sum_{i=1}^N x_0^i \mathds{1}_{P_i}$, $u_t^{[N]}=\sum_{i=1}^N u_t^i \mathds{1}_{P_i}$,
similar to the proof of \cite[Th. 3.14]{lvqi}, we can show that the equation \eqref{yt} admits a unique strong (mild) solution, and thus,
$x_t^{[N]}\in C_\mathcal{F}([0,T];L^2(\Omega,L^2[0,1]))$ is a continuous modification of the unique strong (mild) solution of (\ref{yt}).

Next, we transform social cost $J_{soc}^N(u)$ into a cost function in infinite dimension sense. For $\alpha\in [0,1]$, it holds that
\begin{align*}
  \big((I-\Gamma M^{[N]})x_t^{[N]}\big)(\alpha)&=\sum_{i=1}^N \mathds{1}_{P_i}(\alpha)x_t^i-\Gamma\int_0^1 M^{[N]}(\alpha,\beta) x_t^{[N]}(\beta)d\beta \\
  &=\sum_{i=1}^N \mathds{1}_{P_i}(\alpha)x_t^i-\Gamma\int_0^1 \Big(\sum_{i=1}^N \sum_{j=1}^N \mathds{1}_{P_i}(\alpha) \mathds{1}_{P_j}(\beta) m_{ij}\Big) \Big(\sum_{j=1}^N \mathds{1}_{P_j}(\beta)x_t^j\Big)d\beta\\
  &=\sum_{i=1}^N \mathds{1}_{P_i}(\alpha)\Big[x_t^i-\Gamma\frac{1}{N}\sum_{j=1}^N m_{ij} x_t^j\Big]
\end{align*}
and so
\begin{align}\label{(I-gammaMN)xt[N]}
  \|(I-\Gamma M^{[N]})x_t^{[N]}\|_2^2&=\int_0^1 \Big|\sum_{i=1}^N \mathds{1}_{P_i}(\alpha)\Big[x_t^i-\Gamma\frac{1}{N}\sum_{j=1}^N m_{ij} x_t^j\Big]\Big|^2d\alpha \nonumber\\
   &=\int_0^1 \sum_{i=1}^N \mathds{1}_{P_i}(\alpha) \Big|x_t^i-\Gamma\frac{1}{N}\sum_{j=1}^N m_{ij} x_t^j\Big|^2d\alpha\nonumber\\
   &=\frac{1}{N} \sum_{i=1}^N \Big|x_t^i-\Gamma\frac{1}{N}\sum_{j=1}^N m_{ij} x_t^j\Big|^2.
\end{align}
Similarly, one has
\begin{equation}\label{ut[N]xT[N]}
    \|u_t^{[N]}\|_2^2=\int_0^1 |u_t^{[N]}(\alpha)|^2 d\alpha=\frac{1}{N}\sum_{i=1}^N |u_t^i|^2, \quad
  \|x_T^{[N]}\|_2^2=\int_0^1 |x_T^{[N]}(\alpha)|^2 d\alpha=\frac{1}{N}\sum_{i=1}^N |x_T^i|^2.
\end{equation}
According to the definition of uniqueness for the mild solutions of SEE \cite{Prato}, together with (\ref{ithcost}), (\ref{(I-gammaMN)xt[N]}) and (\ref{ut[N]xT[N]}), we have
\begin{align}\label{J(u[N])=1/NJsoc}
 J(u^{[N]})&\triangleq\frac{1}{2}E\int_0^T \big[Q\|(I-\Gamma M^{[N]})y_t\|_2^2+R\|u_t^{[N]}\|_2^2\big]dt+\frac{1}{2}EQ_T\|y_T\|_2^2 \nonumber\\
   &=\frac{1}{2}E\int_0^T \big[Q\|(I-\Gamma M^{[N]})x_t^{[N]}\|_2^2+R\|u_t^{[N]}\|_2^2\big]dt+\frac{1}{2}EQ_T\|x_T^{[N]}\|_2^2 \nonumber\\
   &=\frac{1}{N}\sum_{i=1}^N\Big\{\frac{1}{2}E\int_0^T\Big[Q\Big(x_t^i-\Gamma\frac{1}{N}\sum\limits_{j=1}^N m_{ij}x_t^j\Big)^2+R(u_t^i)^2\Big]dt+\frac{1}{2}EQ_T|x_T^i|^2\Big\}\nonumber\\
   &=\frac{1}{N}J_{soc}^N(\mathbf{u}).
\end{align}

Define $\mathcal{U}^{step}=\{u^{[N]}|u^{[N]}=\sum_{i=1}^N \mathds{1}_{P_i}u^i,\ u^i\in\mathcal{U}_c\}$. From the above discussion, it follows that the optimization Problem (P1) is equal to the following infinite dimensional optimal control problem.
\begin{problem(P1a)}
Find $\bar{u}^{[N]}\in\mathcal{U}^{step}$ such that
$$
J(\bar{u}^{[N]})=\inf_{u^{[N]}\in\mathcal{U}^{step}}J(u^{[N]})=\inf_{u^{[N]}\in\mathcal{U}^{step}} \Big\{\frac{1}{2}E\int_0^T \big[Q\|(I-\Gamma M^{[N]})y_t\|_2^2+R\|u_t^{[N]}\|_2^2\big]dt+\frac{1}{2}EQ_T\|y_T\|_2^2\Big\}
$$
subject to the following SEE
$$
dy_t=\big[(AI+bM^{[N]})y_t+Bu_t^{[N]}\big]dt+\sigma dW_t^{[N]}, \quad y_0=x_0^{[N]}=\sum_{i=1}^N x_0^i \mathds{1}_{P_i}.
$$
\end{problem(P1a)}
To solve Problem (P1a), we first consider an infinite dimensional control problem in a ``larger'' admissible control set. Define $\mathcal{U}=\{u|u\in L_\mathcal{F}^2(0,T;L^2[0,1])\}$, and we have $\mathcal{U}^{step}\subset\mathcal{U}$.
\begin{problem(P1b)}
Find $\bar{u}\in\mathcal{U}$ such that
$$
J(\bar{u})=\inf_{u\in\mathcal{U}}J(u)=\inf_{u\in\mathcal{U}} \Big\{\frac{1}{2}E\int_0^T \big[Q\|(I-\Gamma M^{[N]})y_t\|_2^2+R\|u_t\|_2^2\big]dt+\frac{1}{2}EQ_T\|y_T\|_2^2\Big\}
$$
subject to the following SEE
$$
dy_t=\big[(AI+bM^{[N]})y_t+Bu_t\big]dt+\sigma dW_t^{[N]}, \quad y_0=x_0^{[N]}=\sum_{i=1}^N x_0^i \mathds{1}_{P_i}.
$$
\end{problem(P1b)}

\begin{proposition}\label{propcentralized}
The optimal control law of Problem (P1b) is given by
\begin{equation}\label{centralizedcontrol}
  \bar{u}_t=-\frac{2B}{R}\Pi_t^{[N]}\bar{y}_t,
\end{equation}
where $\Pi^{[N]}$ satisfies
\begin{equation}\label{PiN}
  \left\{\begin{split}
     &0=\langle \frac{d}{dt}\Pi_t^{[N]}x,x\rangle+2\langle \Pi_t^{[N]}x,(AI+bM^{[N]})x\rangle-\langle \frac{2B^2}{R}(\Pi_t^{[N]})^2x,x\rangle+\langle \frac{1}{2}\widetilde{G}^{[N]}x,x\rangle,\quad x\in L^2[0,1],  \\
      &\Pi_T^{[N]}=\frac{1}{2}Q_TI,
  \end{split}\right.
\end{equation}
with $\widetilde{G}^{[N]}\triangleq Q(I-\Gamma M^{[N]})^*(I-\Gamma M^{[N]})$, and the optimal state is described by
\begin{equation}\label{centralizedstate}
  d\bar{y}_t=(AI+bM^{[N]}-\frac{2B^2}{R}\Pi_t^{[N]})\bar{y}_tdt+\sigma dW_t^{[N]}, \quad \bar{y}_0=x_0^{[N]}=\sum_{i=1}^N x_0^i \mathds{1}_{P_i}.
\end{equation}
\end{proposition}

\begin{proof}
Since
\begin{align*}
  J(u)&=\frac{1}{2}E\int_0^T \big[Q\|(I-\Gamma M^{[N]})y_t\|_2^2+R\|u_t\|_2^2\big]dt+\frac{1}{2}EQ_T\|y_T\|_2^2 \\
   &=\frac{1}{2}E\int_0^T \big[\langle \widetilde{G}^{[N]}y_t,y_t\rangle+\langle RIu_t,u_t\rangle\big]dt+\frac{1}{2}E\langle Q_TIy_T,y_T\rangle,
\end{align*}
from Theorem 4.1 in \cite{Ichikawa}, one has that $\bar{u}_t=-\frac{2B}{R}\Pi_t^{[N]}\bar{y}_t$ is the optimal control law of Problem (P1b), where $\Pi^{[N]}$ satisfies Riccati equation (\ref{PiN}). By Theorem 2.1 and Proposition 2.1 in \cite{Bensoussan2}, (\ref{PiN}) has a unique weak solution $\Pi^{[N]}\in C_\mathcal{S}([0,T];\Sigma(L^2[0,1]))$.
\end{proof}
Next, we prove that $\bar{u}_t=-\frac{2B}{R}\Pi_t^{[N]}\bar{y}_t$ is in $\mathcal{U}^{step}$. For any given graphon operator $M$, we first propose the decomposed form of solution of the following operator-valued Riccati equation
\begin{equation}\label{Pt}
\left\{
  \begin{split}
  &\frac{d}{dt}P_t=-(2AI+2bM)P_t+\frac{2B^2}{R}(P_t)^2-\frac{1}{2}Q(I-\Gamma M)^*(I-\Gamma M),\\
      & P_T=\frac{1}{2}Q_TI.
  \end{split}
  \right.
\end{equation}
\begin{proposition}\label{decomposedprop}
For any given graphon operator $M$ with  orthonormal eigenfunctions $\{f_l^M\}_{l=1}^\infty$ associated with all non-zero eigenvalues $\{\lambda_l^M\}_{l=1}^\infty$, let $\mathcal{S}=(\ker M)^\bot\subset L^2[0,1]$, $\mathcal{S}^\bot$ be orthogonal complement of the former, and let $I_\mathcal{S}$, $I_{\mathcal{S}^\bot}$ be projections on $\mathcal{S}$, $\mathcal{S}^\bot$, respectively. Then $\Pi_t$, given as follows, is the unique weak solution of (\ref{Pt}):
\begin{equation}\label{decomposedform}
  \Pi_t v=\Pi_t^\bot I_{\mathcal{S}^\bot}v+\sum_{l=1}^\infty \bar{\Pi}_t^l (f_l^M\otimes f_l^M)v=\Pi_t^\bot Iv+\sum_{l=1}^\infty (\bar{\Pi}_t^l-\Pi_t^\bot)(f_l^M\otimes f_l^M)v,\quad \forall t\in [0,T],\ v\in L^2[0,1],
\end{equation}
where $(f_l^M\otimes f_l^M)v\triangleq \langle f_l^M,v\rangle f_l^M$, $\{\bar{\Pi}_t^l\}_{l=1}^\infty$ and $\Pi_t^\bot$ are, respectively, the solutions of the following real valued Riccati equations
\begin{equation}\label{Pil}
\left\{
  \begin{split}
     & \frac{d}{dt}\bar{\Pi}_t^l+(2A+2b\lambda_l^M)\bar{\Pi}_t^l-\frac{2B^2}{R}(\bar{\Pi}_t^l)^2+\frac{1}{2}Q(1-\Gamma \lambda_l^M)^2=0, \\
      & \bar{\Pi}_T^l=\frac{1}{2}Q_T
  \end{split}
  \right.
\end{equation}
and
\begin{equation}\label{Pibot}
\left\{
\begin{split}
 & \frac{d}{dt}\Pi_t^\bot+2A\Pi_t^\bot-\frac{2B^2}{R}(\Pi_t^\bot)^2+\frac{1}{2}Q=0,\\
    & \Pi_T^\bot=\frac{1}{2}Q_T.
\end{split}
\right.
\end{equation}
\end{proposition}

\begin{proof}
Let $\tilde{\Pi}^l$ be the solution of the following linear equation
\begin{equation*}
  \frac{d}{dt}\tilde{\Pi}_t^l+(2A+2b\lambda_l^M)\tilde{\Pi}_t^l+\frac{1}{2}Q(1-\Gamma \lambda_l^M)^2=0, \quad \tilde{\Pi}_T^l=\frac{1}{2}Q_T.
\end{equation*}
Applying \cite[Section 8, Th 2]{Sauvigny} and \cite[Lem 7]{Gao1} yields $|\lambda_l^M|\leq 1$, and thus,
there exists a constant $C$ such that $\sup_{1\leq l\leq\infty}\sup_{0\leq t\leq T} |\tilde{\Pi}_t^l|^2\leq C$. By \cite[Th 3.5]{Freiling}, one has $0\leq \bar{\Pi}_t^l\leq \tilde{\Pi}_t^l\leq C$ for all $l$. For any given $v\in L^2[0,1]$, denote $F_n(t)=\sum_{l=1}^n F^l(t)\langle v,f_l^M\rangle f_l^M,\ n=1,2,\cdots$, where $F^l(t)=(2A+2b\lambda_l^M)\bar{\Pi}_t^l-\frac{2B^2}{R}(\bar{\Pi}_t^l)^2+\frac{1}{2}Q(1-\Gamma \lambda_l^M)^2$. By Bessel inequality and boundedness of $\bar{\Pi}_t^l$, there is a constant $C$, such that for each $n$ and $t$,
\begin{equation}\label{dominated}
  \|F_n(t)\|_2^2=\sum_{l=1}^n \Big|F^l(t)\langle v,f_l^M\rangle\Big|^2\leq C\sum_{l=1}^n |\langle v,f_l^M\rangle|^2\leq C\|v\|_2^2,
\end{equation}
which implies the boundedness of $F_n(t)$.
From (\ref{dominated}), we also have $\sum_{l=1}^n |F^l(t)\langle v,f_l^M\rangle|^2$ converges for all $t$. For $n>m$, one has $\|F_n(t)-F_m(t)\|_2^2=\sum_{j=m+1}^n |F^l(t)\langle v,f_l^M\rangle|^2$. Then $\{F_n(t)\}_{n=1}^\infty$ is a Cauchy sequence in $L^2[0,1]$. By the completeness of $L^2[0,1]$, one has $\{F_n(t)\}_{n=1}^\infty$ converges in $L^2[0,1]$ for all $t$, and $\sum_{l=1}^\infty F^l(t)\langle v,f_l^M\rangle f_l^M\in \overline{span\{f_1^M,f_2^M,\cdots\}}$. Then, by (\ref{Pil}) and the dominated convergence theorem for Bochner integrals \cite[Th. 2.18]{lvqi}, we have
\begin{align}\label{sumPil}
  &\sum_{l=1}^\infty \bar{\Pi}_t^l\langle v,f_l^M\rangle f_l^M \nonumber\\
  =&\frac{1}{2}Q_T \sum_{l=1}^\infty \langle v,f_l^M\rangle f_l^M+\sum_{l=1}^\infty \int_t^T \Big[(2A+2b\lambda_l^M)\bar{\Pi}_s^l-\frac{2B^2}{R}(\bar{\Pi}_s^l)^2+\frac{1}{2}Q(1-\Gamma \lambda_l^M)^2\Big]\langle v,f_l^M\rangle f_l^M ds\nonumber\\
  =&\frac{1}{2}Q_T \sum_{l=1}^\infty \langle v,f_l^M\rangle f_l^M+\int_t^T\sum_{l=1}^\infty \Big[(2A+2b\lambda_l^M)\bar{\Pi}_s^l-\frac{2B^2}{R}(\bar{\Pi}_s^l)^2+\frac{1}{2}Q(1-\Gamma \lambda_l^M)^2\Big]\langle v,f_l^M\rangle f_l^M ds\nonumber\\
  =&\frac{1}{2}Q_TI_\mathcal{S}v+\int_t^T \Big[2A\sum_{l=1}^\infty \bar{\Pi}_s^l\langle v,f_l^M\rangle f_l^M+ 2bM\sum_{l=1}^\infty \bar{\Pi}_s^l\langle v,f_l^M\rangle f_l^M-\frac{2B^2}{R}\sum_{l=1}^\infty (\bar{\Pi}_s^l)^2\langle v,f_l^M\rangle f_l^M\nonumber\\
  &+\frac{1}{2}Q(I_\mathcal{S}v-2\Gamma Mv+\Gamma^2M^2v)\Big]ds,\quad \forall t\in [0,T],\ v\in L^2[0,1].
\end{align}
From (\ref{Pibot}),
\begin{equation}\label{PibotISbot}
  \Pi_t^\bot I_{\mathcal{S}^\bot}v=\frac{1}{2}Q_TI_{\mathcal{S}^\bot}v+\int_t^T \Big[2A\Pi_s^\bot I_{\mathcal{S}^\bot}v-\frac{2B^2}{R}(\Pi_s^\bot)^2I_{\mathcal{S}^\bot}v+\frac{1}{2}QI_{\mathcal{S}^\bot}v\Big]ds,\quad \forall t\in [0,T],\ v\in L^2[0,1].
\end{equation}
Since
$$(\Pi_t)^2v=\Pi_t(\Pi_tv)=\sum_{l=1}^\infty \bar{\Pi}_t^l \langle \Pi_tv,f_l^M\rangle f_l^M+\Pi_t^\bot I_{\mathcal{S}^\bot}(\Pi_tv)=\sum_{l=1}^\infty (\bar{\Pi}_t^l)^2 \langle v,f_l^M\rangle f_l^M+(\Pi_t^\bot)^2I_{\mathcal{S}^\bot}v,$$
combining (\ref{sumPil}) and (\ref{PibotISbot}) yields
\begin{align}\label{Pitv}
  \Pi_tv & =\sum_{l=1}^\infty \bar{\Pi}_t^l \langle v,f_l^M\rangle f_l^M+\Pi_t^\bot I_{\mathcal{S}^\bot}v \nonumber\\
   & =\frac{1}{2}Q_Tv+\int_t^T \Big[(2AI+2bM)\Pi_s-\frac{2B^2}{R}(\Pi_s)^2+\frac{1}{2}Q(I-\Gamma M)^2\Big]vds,\quad \forall t\in [0,T],\ v\in L^2[0,1].
\end{align}

Next we prove that $\Pi_t$ given by (\ref{decomposedform}) is in $C_{\mathcal{S}}([0,T];\Sigma (L^2[0,1]))$. Indeed, for all $t\in [0,T]$, $v\in L^2[0,1]$, there is a constant $C$ such that
\begin{align*}
   \|\Pi_tv\|^2=\|\Pi_t^\bot Iv+\sum_{l=1}^\infty (\bar{\Pi}_t^l-\Pi_t^\bot)\langle f_l^M,v\rangle f_l^M\|^2
   \leq 2|\Pi_t^\bot|^2\|v\|^2+2\sum_{l=1}^\infty |\bar{\Pi}_t^l-\Pi_t^\bot|^2|\langle f_l^M,v\rangle|^2\leq C\|v\|^2.
\end{align*}
Then $\Pi_t$ is bounded. For any $u, v\in L^2[0,1]$, one has
\begin{align*}
  \langle \Pi_tv,u\rangle & =\Big\langle\Pi_t^\bot Iv+\sum_{l=1}^\infty (\bar{\Pi}_t^l-\Pi_t^\bot)\langle f_l^M,v\rangle f_l^M, u\Big\rangle\\
  & =\langle \Pi_t^\bot Iv,u\rangle+\sum_{l=1}^\infty (\bar{\Pi}_t^l-\Pi_t^\bot)\langle f_l^M,v\rangle \langle f_l^M, u\rangle\\
  & =\langle v,\Pi_t^\bot Iu\rangle+\Big\langle v,\sum_{l=1}^\infty (\bar{\Pi}_t^l-\Pi_t^\bot)\langle f_l^M,u\rangle f_l^M\Big\rangle\\
  & =\langle v,\Pi_tu\rangle.
\end{align*}
Then $\Pi_t$ is self-adjoint. Since $\Pi_t$ satisfies (\ref{Pitv}), one has that $\Pi_t$ is strongly continuous. Then $\Pi_t$ given by (\ref{decomposedform}) is in $C_{\mathcal{S}}([0,T];\Sigma (L^2[0,1]))$. From (\ref{Pitv}) and Proposition 3.3 in \cite{lvqi1}, $\Pi_t\in C_{\mathcal{S}}([0,T];\Sigma (L^2[0,1]))$ is a weak solution of (\ref{Pt}). By Theorem 2.1 and Proposition 2.1 in \cite{Bensoussan2}, $\Pi_t$ is the unique weak solution.
\end{proof}

\begin{remark} A class of operator-valued Riccati equations similar to \eqref{Pt} has been considered by Gao et al. \cite{Gao5}, but the proof method of Proposition \ref{decomposedprop} is quite different from the one used 
in \cite[Th. 1]{Gao5}.
\end{remark}
Obviously, the conclusion in Proposition \ref{decomposedprop} still holds for graphon operator with finite rank.
Let $rank\ M_N=r_N$ and $rank\ Q_N=d_N$ with $1\leq r_N, d_N\leq N$. By Lemma \ref{buchongGaoprop} and Proposition \ref{decomposedprop}, the optimal control (\ref{centralizedcontrol}) can also be presented as
\begin{align}\label{decomposedcontrol}
  \bar{u}_t & =-\frac{2B}{R}\Pi_t^\bot I\bar{y}_t-\frac{2B}{R}\sum_{l=1}^{r_N}(\bar{\Pi}_t^{Nl}-\Pi_t^\bot)\langle \bar{y}_t,f_l^{M^{[N]}}\rangle f_l^{M^{[N]}}\nonumber\\
  & =-\frac{2B}{R}\Pi_t^\bot I\bar{y}_t-\frac{2B}{R}\sqrt{N} \sum_{k=1}^N \Big[\sum_{l=1}^{r_N} (\bar{\Pi}_t^{Nl}-\Pi_t^\bot)\langle \bar{y}_t,f_l^{M^{[N]}}\rangle v_l^{M_N}(k)\Big]\mathds{1}_{P_k},
\end{align}
where $f_l^{M^{[N]}}=\sqrt{N}S_{v_l^{M_N}}=\sqrt{N}\sum_{k=1}^N \mathds{1}_{P_k} v_l^{M_N}(k)$, and $\bar{\Pi}_t^{Nl}$ satisfies
\begin{equation*}
\left\{
  \begin{split}
     & \frac{d}{dt}\bar{\Pi}_t^{Nl}+(2A+\frac{2b}{N}\lambda_l^{M_N})\bar{\Pi}_t^{Nl}-\frac{2B^2}{R}(\bar{\Pi}_t^{Nl})^2+\frac{1}{2}Q(1-\frac{\Gamma}{N}\lambda_l^{M_N})^2=0, \\
      & \bar{\Pi}_T^{Nl}=\frac{1}{2}Q_T.
   \end{split}
   \right.
\end{equation*}
Let $\varphi_t^{Nl}=\langle \bar{y}_t,f_l^{M^{[N]}}\rangle$. Then it follows from (\ref{centralizedstate}) that for each $t\in[0,T]$,
\begin{equation*}
  \langle \bar{y}_t,f_l^{M^{[N]}}\rangle=\langle x_0^{[N]},f_l^{M^{[N]}}\rangle+\int_0^t \Big\langle (AI+bM^{[N]}-\frac{2B^2}{R}\Pi_s^{[N]})\bar{y}_s,f_l^{M^{[N]}}\Big\rangle ds+\sigma\langle W_t^{[N]},f_l^{M^{[N]}}\rangle, \quad a.s.
\end{equation*}
Since
$$
\langle M^{[N]}\bar{y}_t,f_l^{M^{[N]}}\rangle=\lambda_l^{M^{[N]}}\langle \bar{y}_t,f_l^{M^{[N]}}\rangle=\frac{1}{N}\lambda_l^{M_N}\langle \bar{y}_t,f_l^{M^{[N]}}\rangle,
$$
$$
\langle\Pi_t^{[N]}\bar{y}_t,f_l^{M^{[N]}}\rangle=\Big\langle \Pi_t^\bot I\bar{y}_t+\sum_{k=1}^{r_N}(\bar{\Pi}_t^{Nk}-\Pi_t^\bot)\langle \bar{y}_t,f_k^{M^{[N]}}\rangle f_k^{M^{[N]}},f_l^{M^{[N]}}\Big\rangle=\bar{\Pi}_t^{Nl}\langle \bar{y}_t,f_l^{M^{[N]}}\rangle,
$$
and
\begin{align*}
  \langle W_t^{[N]},f_l^{M^{[N]}}\rangle & =\Big\langle \sum_{j=1}^{d_N}\sqrt{\frac{\lambda_j^{Q_N}}{N}}\sqrt{N}S_{v_j^{Q_N}}W_t^j,\sqrt{N}S_{v_l^{M_N}}\Big\rangle=\sum_{j=1}^{d_N}\sqrt{N\lambda_j^{Q_N}}\langle S_{v_j^{Q_N}},S_{v_l^{M_N}}\rangle W_t^j\nonumber\\
  & =\sum_{j=1}^{d_N}\sqrt{N\lambda_j^{Q_N}}\Big\langle \sum_{k=1}^N\mathds{1}_{P_k}v_j^{Q_N}(k),\sum_{k=1}^N\mathds{1}_{P_k}v_l^{M_N}(k)\Big\rangle W_t^j=\sum_{j=1}^{d_N}\sqrt{\frac{\lambda_j^{Q_N}}{N}}\langle v_j^{Q_N},v_l^{M_N}\rangle W_t^j,
\end{align*}
we have for $t\in[0,T]$,
\begin{equation}\label{varphiNl}
  \varphi_t^{Nl}=\langle x_0^{[N]},f_l^{M^{[N]}}\rangle+\int_0^t(A+\frac{b}{N}\lambda_l^{M_N}-\frac{2B^2}{R}\bar{\Pi}_s^{Nl})\varphi_s^{Nl}ds+\sigma \sum_{j=1}^{d_N}\sqrt{\frac{\lambda_j^{Q_N}}{N}}\langle v_j^{Q_N},v_l^{M_N}\rangle W_t^j, \quad a.s.
\end{equation}
and so  (\ref{decomposedcontrol}) has the following modification
\begin{equation}\label{decomposedcontrol1}
  \bar{u}_t=-\frac{2B}{R}\Pi_t^\bot I\bar{y}_t-\frac{2B}{R}\sqrt{N} \sum_{k=1}^N \Big[\sum_{l=1}^{r_N} (\bar{\Pi}_t^{Nl}-\Pi_t^\bot)\breve{\varphi}_t^{Nl} v_l^{M_N}(k)\Big]\mathds{1}_{P_k}.
\end{equation}
where $\breve{\varphi}_t^{Nl}$ is a continuous solution of \eqref{varphiNl}. By a slight abuse of notation, we still use $\bar{u}_t$ to denote the modification of control in (\ref{decomposedcontrol}). We note that $\bar{u}_t$ given in \eqref{decomposedcontrol1} is also an optimal control for Problem (P1b), and the corresponding state satisfies for $t\in[0,T]$
\begin{equation*}
  d\bar{y}_t=\Big\{\Big[\Big(A-\frac{2B^2}{R}\Pi_t^\bot\Big)I+bM^{[N]}\Big]\bar{y}_t-\frac{2B^2}{R}\sqrt{N} \sum_{k=1}^N \Big[\sum_{l=1}^{r_N} (\bar{\Pi}_t^{Nl}-\Pi_t^\bot)\breve{\varphi}_t^{Nl} v_l^{M_N}(k)\Big]\mathds{1}_{P_k}\Big\}dt+\sigma dW_t^{[N]},\quad a.s.
\end{equation*}
where $\sum_{l=1}^{r_N} (\bar{\Pi}_t^{Nl}-\Pi_t^\bot)\breve{\varphi}_t^{Nl} v_l^{M_N}(k)\in L_\mathcal{F}^2(0,T;\mathbb{R})$, $k=1,\cdots,N$. Thus, similar to the previous derivation of \eqref{xt[N]}-\eqref{yt}, we can obtain that $\bar{x}_t^{[N]}\triangleq \sum_{i=1}^N \bar{x}_t^i \mathds{1}_{P_i}$ is a continuous modification of $\bar{y}_t$, where $\bar{x}_t^i$ is the unique solution of the following equation
\begin{equation}\label{centralizedstateith}
  \bar{x}_t^i=x_0^i+\int_0^t \Big[(A-\frac{2B^2}{R}\Pi_s^\bot)\bar{x}_s^i+\frac{b}{N}\sum_{j=1}^N m_{ij}\bar{x}_s^j-\frac{2B^2}{R}\sqrt{N}\sum_{l=1}^{r_N} (\bar{\Pi}_s^{Nl}-\Pi_s^\bot)\breve{\varphi}_s^{Nl} v_l^{M_N}(i)\Big]ds+\sigma \widetilde{W}_t^i.
\end{equation}
It follows from \eqref{decomposedcontrol1}-\eqref{centralizedstateith} that $\bar{u}_t$ has the following modification
\begin{equation}\label{decomposedcontrol2}
  \bar{u}_t=-\frac{2B}{R}\sum_{i=1}^N \Big\{\Pi_t^\bot \bar{x}_t^i+\sqrt{N}\Big[\sum_{l=1}^{r_N} (\bar{\Pi}_t^{Nl}-\Pi_t^\bot)\breve{\varphi}_t^{Nl} v_l^{M_N}(i)\Big]\Big\}\mathds{1}_{P_i},
\end{equation}
which is optimal in $\mathcal{U}$ and $\bar{u}\in\mathcal{U}^{step}$, which means that $\bar{u}$ given by \eqref{decomposedcontrol2} is also the optimal control for Problem (P1a). For convenience, we denote $\bar{u}_t^{[N]}\triangleq \bar{u}_t$. It follows from \eqref{decomposedcontrol2} that
\begin{equation}\label{centralizedcontrolrewrite}
  \bar{u}_t^{[N]}=-\frac{2B}{R}\Pi_t^{[N]}\bar{x}_t^{[N]}.
\end{equation}
By \eqref{decomposedcontrol2} and the equivalence between Problem (P1) and (P1a), the feedback type centralized control for $i$th agent is given as follows
\begin{align}\label{centralizedcontrolith}
  \bar{u}_t^i & =\bar{u}_t^{[N]}(\alpha),\quad \alpha\in P_i\nonumber\\
  & =-\frac{2B}{R}\Big\{\Pi_t^\bot \bar{x}_t^i+\sqrt{N}\Big[\sum_{l=1}^{r_N} (\bar{\Pi}_t^{Nl}-\Pi_t^\bot)\breve{\varphi}_t^{Nl} v_l^{M_N}(i)\Big]\Big\},
\end{align}
where $\bar{x}_t^i$ given by (\ref{centralizedstateith}) is the optimal state. We note that $\bar{u}_t^i$ depends on $M_N$, $Q_N$ and the information of $\mathcal{F}_t$.

\section{Asymptotically optimal decentralized controls}
\qquad The following two assumptions play a key role for designing the asymptotically optimal decentralized controls.
\begin{assumption}\label{assumption1}
There exist a graphon $M$ and $x_0\in L^2[0,1]$ such that, as $N\rightarrow \infty$,
$$
\|M^{[N]}-M\|_{op}\rightarrow 0,\quad \|x_0^{[N]}-x_0\|_2\rightarrow 0.
$$
\end{assumption}
\begin{assumption}\label{assumption2}
There is a $Q$-Wiener process $W_t^Q=\sum_{j=1}^d \sqrt{\lambda_j^Q} W_t^j f_j^Q$ such that, as $N\rightarrow \infty$,
\begin{equation*}
  \sum_{j=1}^d \Big\|\sqrt{\lambda_j^Q}f_j^Q-\sqrt{\lambda_j^{Q_N}}\sum_{i=1}^N v_j^{Q_N}(i) \mathds{1}_{P_i}\Big\|^2\rightarrow 0,\quad \sum_{j=d+1}^{d_N} \Big\|\sqrt{\lambda_j^{Q_N}}\sum_{i=1}^N v_j^{Q_N}(i) \mathds{1}_{P_i}\Big\|^2=\frac{1}{N} \sum_{j=d+1}^{d_N} \lambda_j^{Q_N}\rightarrow 0.
\end{equation*}
\end{assumption}
Define $\mathcal{F}_t^d=\sigma (W_s^k, 0\leq s\leq t, 1\leq k\leq d)\vee\mathcal{N}$ and $\mathcal{U}^Q=\{u|u\in L_{\mathcal{F}^d}^2(0,T;L^2[0,1])\}$. To design asymptotically optimal decentralized controls, we first consider an auxiliary infinite dimensional control problem.
\begin{problem(P2a)}
Find $\tilde{u}\in\mathcal{U}^Q$ such that $\tilde{J}(\tilde{u})=\inf_{u\in\mathcal{U}^Q}\tilde{J}(u)$, where
$$
\tilde{J}(u)=\frac{1}{2}E\int_0^T \big[Q\|(I-\Gamma M)y_t\|^2+R\|u_t\|^2\big]dt+\frac{1}{2}EQ_T\|y_T\|^2,
$$
subject to the following SEE
\begin{equation*}
  dy_t=[(AI+bM)y_t+Bu_t]dt+\sigma dW_t^Q,\quad y_0=x_0.
\end{equation*}
\end{problem(P2a)}
Similar to the proof of Proposition \ref{propcentralized},  we can show that the optimal control law of Problem (P2a) is given  by
\begin{equation}\label{tildeu}
  \tilde{u}_t=-\frac{2B}{R}\Pi_t\tilde{y}_t,
\end{equation}
where $\Pi_t$ satisfies
\begin{equation}\label{Pi}
  \left\{\begin{split}
     &0=\langle \frac{d}{dt}\Pi_tx,x\rangle+2\langle \Pi_tx,(AI+bM)x\rangle-\langle \frac{2B^2}{R}(\Pi_t)^2x,x\rangle+\langle \frac{1}{2}Q(I-\Gamma M)^*(I-\Gamma M)x,x\rangle,\quad x\in L^2[0,1],  \\
      &\Pi_T=\frac{1}{2}Q_TI,
  \end{split}\right.
\end{equation}
and the optimal state is described by
\begin{equation}\label{auxiliarystate}
  d\tilde{y}_t=(AI+bM-\frac{2B^2}{R}\Pi_t)\tilde{y}_tdt+\sigma dW_t^Q,\quad \tilde{y}_0=x_0.
\end{equation}
From Proposition \ref{decomposedprop}, we have
\begin{align}\label{auxiliarycontrol}
  \tilde{u}_t =-\frac{2B}{R}\Pi_t\tilde{y}_t
  =-\frac{2B}{R}\Pi_t^\bot I\tilde{y}_t-\frac{2B}{R}\sum_{l=1}^\infty (\bar{\Pi}_t^l-\Pi_t^\bot)\langle \tilde{y}_t,f_l^M\rangle f_l^M,
\end{align}
where $\{\bar{\Pi}_t^l\}_{l=1}^\infty$, $\Pi_t^\bot$ are real valued Riccati equations satisfy (\ref{Pil}) and (\ref{Pibot}) respectively. By \cite[Th. 7.2]{Prato}, $\tilde{y}$ possesses a continuous modification. In the sequel, we shall always use such a modification for $\tilde{y}$. Denote $\varphi_t^l=\langle \tilde{y}_t,f_l^M\rangle$.
Based on (\ref{auxiliarycontrol}), we can construct the following control for $i$th agent
\begin{equation}\label{constructedcontrol}
  \hat{u}_t^i=-\frac{2B}{R}\Pi_t^\bot\hat{x}_t^i-\frac{2B}{R}N \Big\langle \mathds{1}_{P_i},\sum_{l=1}^\infty (\bar{\Pi}_t^l-\Pi_t^\bot)\varphi_t^lf_l^M\Big\rangle
\end{equation}
and the corresponding state satisfies
\begin{equation*}
  \hat{x}_t^i=x_0^i+\int_0^t \Big[(A-\frac{2B^2}{R}\Pi_s^\bot)\hat{x}_s^i+b\frac{1}{N}\sum_{j=1}^N m_{ij}\hat{x}_s^j-\frac{2B^2}{R}N \Big\langle \mathds{1}_{P_i},\sum_{l=1}^\infty (\bar{\Pi}_s^l-\Pi_s^\bot)\varphi_s^lf_l^M\Big\rangle\Big] ds+\sigma\widetilde{W}_t^i.
\end{equation*}
To verify $\hat{u}^i\in\mathcal{U}_{sd}$, it is enough to prove $\big\langle \mathds{1}_{P_i},\sum_{l=1}^\infty (\bar{\Pi}_t^l-\Pi_t^\bot)\varphi_t^lf_l^M\big\rangle\in L_{\mathcal{F}^d}^2(0,T;\mathbb{R})\subset L_{\mathcal{G}^i}^2(0,T;\mathbb{R})$. Indeed, by the definition of $\varphi_t^l$, one has that $\varphi_t^l$ is $\mathcal{F}^d$-progressively measurable, i.e., $\varphi_s^l$ is $\mathcal{B}([0,t])\otimes\mathcal{F}_t^d$ measurable for each $t\in [0,T]$. Then for each $t\in [0,T]$, there exists a sequence of $\mathcal{B}([0,t])\otimes\mathcal{F}_t^d$-simple functions converge pointwise to $\varphi_s^l$, which implies that there is also a sequence of $\mathcal{B}([0,t])\otimes\mathcal{F}_t^d$-simple functions converge pointwise to $\varphi_s^l f_l^M$. Thus $\varphi_t^l f_l^M$ is $\mathcal{F}^d$-progressively measurable, which indicates that $\big\langle \mathds{1}_{P_i},\sum_{l=1}^n (\bar{\Pi}_t^l-\Pi_t^\bot)\varphi_t^lf_l^M\big\rangle$ is $\mathcal{F}^d$-progressively measurable for each $n\geq 1$. Similar to the proof in Prop \ref{decomposedprop}, it holds that $\lim\limits_{n\rightarrow\infty}\big\langle \mathds{1}_{P_i},\sum_{l=1}^n (\bar{\Pi}_t^l-\Pi_t^\bot)\varphi_t^lf_l^M\big\rangle=\big\langle \mathds{1}_{P_i},\sum_{l=1}^\infty (\bar{\Pi}_t^l-\Pi_t^\bot)\varphi_t^lf_l^M\big\rangle$ pointwise, then $\big\langle \mathds{1}_{P_i},\sum_{l=1}^\infty (\bar{\Pi}_t^l-\Pi_t^\bot)\varphi_t^lf_l^M\big\rangle$ is $\mathcal{F}^d$-progressively measurable. Since there exists a constant $C$ such that
\begin{align*}
 & E\int_0^T \Big|\Big\langle \mathds{1}_{P_i},\sum_{l=1}^\infty (\bar{\Pi}_t^l-\Pi_t^\bot)\varphi_t^lf_l^M\Big\rangle\Big|^2dt \leq E\int_0^T \|\mathds{1}_{P_i}\|^2 \|\sum_{l=1}^\infty (\bar{\Pi}_t^l-\Pi_t^\bot)\varphi_t^lf_l^M\|^2dt\\
  =&\frac{1}{N}E\int_0^T \sum_{l=1}^\infty |\bar{\Pi}_t^l-\Pi_t^\bot|^2|\langle \tilde{y}_t,f_l^M\rangle|^2dt =\frac{C}{N}E\int_0^T \sum_{l=1}^\infty |\langle \tilde{y}_t,f_l^M\rangle|^2dt
  \le \frac{C}{N}E\int_0^T \|\tilde{y}_t\|^2dt<\infty,
\end{align*}
we have $\big\langle \mathds{1}_{P_i},\sum_{l=1}^\infty (\bar{\Pi}_t^l-\Pi_t^\bot)\varphi_t^lf_l^M\big\rangle\in L_{\mathcal{G}^i}^2(0,T;\mathbb{R})$ and so $\hat{u}^i\in\mathcal{U}_{sd}$.

We are now in a position to propose the result of asymptotic optimality of the decentralized controls \eqref{constructedcontrol}.
\begin{theorem}\label{theoremasymptotic}
Let Assumptions \ref{assumption1}, \ref{assumption2} hold. For Problem (P2), the set of controls $\mathbf{\hat{u}}=(\hat{u}^1,\cdots,\hat{u}^N)$ given by (\ref{constructedcontrol}) has asymptotic social optimality, i.e., $
  \Big|\frac{1}{N}J_{soc}^N(\mathbf{\hat{u}})-\frac{1}{N}\inf_{\mathbf{u}\in\mathcal{U}_c}J_{soc}^N(\mathbf{u})\Big|=o(1).
$
\end{theorem}
To prove Theorem \ref{theoremasymptotic}, we first provide a transformation for social cost $J_{soc}^N(\mathbf{\hat{u}})$. Let $\hat{x}_t^{[N]}=\sum_{i=1}^N \mathds{1}_{P_i}\hat{x}_t^i$. Since $\big\langle \mathds{1}_{P_i},\sum_{l=1}^\infty (\bar{\Pi}_t^l-\Pi_t^\bot)\varphi_t^lf_l^M\big\rangle\in L_{\mathcal{G}^i}^2(0,T;\mathbb{R})\subset L_\mathcal{F}^2(0,T;\mathbb{R})$, one has $\sum_{i=1}^N \mathds{1}_{P_i}\big\langle \mathds{1}_{P_i},\sum_{l=1}^\infty (\bar{\Pi}_t^l-\Pi_t^\bot)\varphi_t^lf_l^M\big\rangle\in L_\mathcal{F}^2(0,T;L^2[0,1])$. Similar to the discussion among (\ref{vectorSDE})-(\ref{J(u[N])=1/NJsoc}), we derive that $\hat{x}_t^{[N]}$ is the unique strong solution of the following SEE
\begin{equation}\label{y_t}
  y_t=x_0^{[N]}+\int_0^t \Big\{\Big[(A-\frac{2B^2}{R}\Pi_s^\bot)I+bM^{[N]}\Big]y_s-\frac{2B^2}{R}N\sum_{i=1}^N\mathds{1}_{P_i}\Big\langle \mathds{1}_{P_i},\sum_{l=1}^\infty (\bar{\Pi}_s^l-\Pi_s^\bot)\varphi_s^lf_l^M\Big\rangle\Big\}ds+\sigma W_t^{[N]}
\end{equation}
and
\begin{align}\label{J(hatu[N])=1/NJsoc}
  \frac{1}{N}J_{soc}^N(\mathbf{\hat{u}}) & =\frac{1}{N}\sum_{i=1}^N\Big\{\frac{1}{2}E\int_0^T\Big[Q\Big(\hat{x}_t^i-\Gamma\frac{1}{N}\sum\limits_{j=1}^N m_{ij}\hat{x}_t^j\Big)^2+R(\hat{u}_t^i)^2\Big]dt+\frac{1}{2}EQ_T|\hat{x}_T^i|^2\Big\}\nonumber\\
  & =\frac{1}{2}E\int_0^T \big[Q\|(I-\Gamma M^{[N]})\hat{x}_t^{[N]}\|_2^2+R\|\hat{u}_t^{[N]}\|_2^2\big]dt+\frac{1}{2}EQ_T\|\hat{x}_T^{[N]}\|_2^2
\end{align}
where
\begin{equation}\label{hatuN}
  \hat{u}_t^{[N]}=\sum_{i=1}^N \mathds{1}_{P_i}\hat{u}_t^i=-\frac{2B}{R}\Pi_t^\bot I\hat{x}_t^{[N]}-\frac{2B}{R}N\sum_{i=1}^N\mathds{1}_{P_i}\big\langle \mathds{1}_{P_i},\sum_{l=1}^\infty (\bar{\Pi}_t^l-\Pi_t^\bot)\varphi_t^lf_l^M\big\rangle.
\end{equation}

Next, we propose several Lemmas before we prove Theorem \ref{theoremasymptotic}.
\begin{lemma}\label{lemmabounded}
There exists a constant $C$, such that $\sup\limits_{1\leq N<\infty} \sup\limits_{t\in[0,T]}\|\mathds{A}_t^{[N]}\|_{op}^2\leq C$, where $\mathds{A}_t^{[N]}=AI-\frac{2B^2}{R}\Pi_t^{[N]}+bM^{[N]}$.
\end{lemma}
\begin{proof}
The proof is similar to the one of \cite[Th. 10]{Gao1}, so we omit it.
\end{proof}

\begin{lemma}\label{lemmaapproximate}
For all $\phi\in L^2[0,1]$, one has $\Big\|N\sum\limits_{i=1}^N\langle \phi,\mathds{1}_{P_i}\rangle\mathds{1}_{P_i}-\phi\Big\|_2\rightarrow 0$ as $N\rightarrow\infty$.
\end{lemma}
\begin{proof}
Define a linear operator $\Upsilon_N$ on $L^2[0,1]$ by letting $\Upsilon_N \phi=N\sum_{i=1}^N\langle \phi,\mathds{1}_{P_i}\rangle\mathds{1}_{P_i}$, $\phi\in L^2[0,1]$. Now, we first show that $\Upsilon_N$ is bounded uniformly. It holds that
\begin{align*}
  \|\Upsilon_N \phi\|_2^2 & =N^2\sum_{i=1}^N|\langle \mathds{1}_{P_i},\phi\rangle|^2\|\mathds{1}_{P_i}\|_2^2 =N\sum_{i=1}^N|\langle \mathds{1}_{P_i},\phi\rangle|^2=N\sum_{i=1}^N\Big|\int_{P_i}\phi(\alpha)d\alpha\Big|^2\\
  &\leq N\sum_{i=1}^N\int_{P_i}1d\alpha \int_{P_i}|\phi(\alpha)|^2d\alpha=\int_0^1 |\phi(\alpha)|^2d\alpha=\|\phi\|_2^2,
\end{align*}
and so $\|\Upsilon_N\|_{op}\leq 1$ for all $N$. Next, we deduce that the result is true for $\phi\in C[0,1]$. Indeed, for each $\phi\in C[0,1]$, we have
\begin{align*}
  \|\Upsilon_N\phi-\phi\|_2^2 &=\Big\|N\sum_{i=1}^N\langle \phi,\mathds{1}_{P_i}\rangle\mathds{1}_{P_i}-\phi\Big\|_2^2\\
  & =\int_0^1\Big|\sum_{i=1}^N \mathds{1}_{P_i}(\theta)\Big(\phi(\theta)-N\int_{P_i}\phi(\alpha)d\alpha\Big)\Big|^2d\theta\\
  & =\int_0^1\sum_{i=1}^N \mathds{1}_{P_i}(\theta)\Big|\phi(\theta)-N\int_{P_i}\phi(\alpha)d\alpha\Big|^2d\theta\\
  & =\int_0^1\sum_{i=1}^N \mathds{1}_{P_i}(\theta)\Big|\phi(\theta)-\phi(\xi_i)\Big|^2d\theta,\quad \xi_i\in P_i\\
  & =\sum_{i=1}^N\int_{P_i}\Big|\phi(\theta)-\phi(\xi_i)\Big|^2d\theta\\
  & =\frac{1}{N}\sum_{i=1}^N\Big|\phi(\zeta_i)-\phi(\xi_i)\Big|^2,\quad \zeta_i\in P_i.
\end{align*}
Since $\phi\in C[0,1]$, $\phi(\cdot)$ is uniformly continuous on $[0,T]$. Thus, $\forall \varepsilon>0$, $\exists N_0$, such that for $N\geq N_0$, one has $\frac{1}{N}\sum_{i=1}^N\Big|\phi(\zeta_i)-\phi(\xi_i)\Big|^2<\varepsilon$. Then, for each $\phi\in C[0,1]$, $\|\Upsilon_N\phi-\phi\|_2\rightarrow 0$, as $N\rightarrow\infty$.

Since $C[0,1]$ is dense in $L^2[0,1]$, for each $\phi\in L^2[0,1]$, there exist a sequence $\{\phi_k\}\subset C[0,1]$ such that $\|\phi_k-\phi\|_2\rightarrow 0$. Then, it follows from the uniform boundedness of $\Upsilon_N$ that
\begin{align*}
  \|\Upsilon_N\phi-\phi\|_2 &=\|\Upsilon_N\phi-\Upsilon_N\phi_k+\Upsilon_N\phi_k-\phi_k+\phi_k-\phi\|_2\\
  &\leq \|\Upsilon_N(\phi-\phi_k)\|_2+\|\Upsilon_N\phi_k-\phi_k\|_2+\|\phi_k-\phi\|_2\\
  &\leq 2\|\phi_k-\phi\|_2+\|\Upsilon_N\phi_k-\phi_k\|_2,\quad k\geq 1.
\end{align*}
For any $\epsilon>0$, we can find $k_0\geq 1$ such that $\|\phi_{k_0}-\phi\|_2<\frac{\epsilon}{4}$. Moreover, $\exists N_1$, for $N\geq N_1$, one has $\|\Upsilon_N\phi_{k_0}-\phi_{k_0}\|_2<\frac{\epsilon}{2}$. Therefore, it holds that $\|\Upsilon_N\phi-\phi\|_2\leq 2\|\phi_{k_0}-\phi\|_2+\|\Upsilon_N\phi_{k_0}-\phi_{k_0}\|_2<\epsilon$. This completes the proof.
\end{proof}

\begin{lemma}\label{lemma3}
There is a constant $C$, such that
$$
\lim_{N\rightarrow\infty}\sup_{t\in [0,T]}E\|\tilde{y}_t-\bar{x}_t^{[N]}\|^2=0, \quad \sup\limits_{1\leq N<\infty} \sup\limits_{t\in[0,T]}E\|\bar{x}_t^{[N]}\|^2\leq C.
$$
\end{lemma}
\begin{proof}
Denote $\mathds{A}_t\triangleq AI-\frac{2B^2}{R}\Pi_t+bM$ and $\mathds{A}_t^{[N]}\triangleq AI-\frac{2B^2}{R}\Pi_t^{[N]}+bM^{[N]}$. From (\ref{centralizedstate}) and (\ref{auxiliarystate}), one has
\begin{equation*}
  \tilde{y}_t-\bar{x}_t^{[N]}=x_0-x_0^{[N]}+\int_0^t (\mathds{A}_s\tilde{y}_s-\mathds{A}_s^{[N]}\bar{x}_s^{[N]})ds+\sigma(W_t^Q-W_t^{[N]}).
\end{equation*}
Applying the triangle inequality leads to
\begin{align*}
  \|\tilde{y}_t-\bar{x}_t^{[N]}\|^2 & \leq 3\|x_0-x_0^{[N]}\|^2+3t\int_0^t\|\mathds{A}_s\tilde{y}_s-\mathds{A}_s^{[N]}\bar{x}_s^{[N]}\|^2ds+3\sigma^2\|W_t^Q-W_t^{[N]}\|^2\nonumber\\
  & \leq 3\|x_0-x_0^{[N]}\|^2+6t\int_0^t\big[\|(\mathds{A}_s-\mathds{A}_s^{[N]})\tilde{y}_s\|^2+\|\mathds{A}_s^{[N]}\|_{op}^2\|\tilde{y}_s-\bar{x}_s^{[N]}\|^2\big]ds+3\sigma^2\|W_t^Q-W_t^{[N]}\|^2\nonumber\\
  & \leq 3\|x_0-x_0^{[N]}\|^2+6T\int_0^T \|(\mathds{A}_s-\mathds{A}_s^{[N]})\tilde{y}_s\|^2ds\nonumber\\
  &\quad+6T\left(\sup\limits_{1\leq N<\infty} \sup\limits_{s\in[0,T]}\|\mathds{A}_s^{[N]}\|_{op}^2\right)\int_0^t\|\tilde{y}_s-\bar{x}_s^{[N]}\|^2ds
  +3\sigma^2\|W_t^Q-W_t^{[N]}\|^2.
\end{align*}
Since $E\|\tilde{y}_t-\bar{x}_t^{[N]}\|^2\leq 2\sup_{t\in[0,T]}E\|\tilde{y}_t\|^2+2\sup_{t\in[0,T]}E\|\bar{x}_t^{[N]}\|^2$, one has that $E\|\tilde{y}_t-\bar{x}_t^{[N]}\|^2$ is bounded on $[0,T]$ for each $N$. Then, taking expectation on both sides and applying Gronwall's inequality, we have
\begin{align}\label{22}
  E\|\tilde{y}_t-\bar{x}_t^{[N]}\|^2\leq &\Big(3\|x_0-x_0^{[N]}\|^2+6TE\int_0^T \|(\mathds{A}_s-\mathds{A}_s^{[N]})\tilde{y}_s\|^2ds\Big.\nonumber\\
  &\Big.+3\sigma^2\sup_{t\in[0,T]}E\|W_t^Q-W_t^{[N]}\|^2\Big)\exp\big(6T^2\sup\limits_{1\leq N<\infty} \sup\limits_{t\in[0,T]}\|\mathds{A}_t^{[N]}\|_{op}^2\big).
\end{align}
Since
\begin{align*}
  E\|W_t^{[N]}-W_t^Q\|^2 & =E\Big\|\sum_{j=1}^d \Big(\sqrt{\lambda_j^Q}f_j^Q-\sqrt{\lambda_j^{Q_N}}\sum_{i=1}^N v_j^{Q_N}(i) \mathds{1}_{P_i}\Big)W_t^j+\sum_{j=d+1}^{d_N} \sqrt{\lambda_j^{Q_N}}\sum_{i=1}^N v_j^{Q_N}(i) \mathds{1}_{P_i}W_t^j\Big\|^2\\
  & =\sum_{j=1}^d\Big\|\sqrt{\lambda_j^Q}f_j^Q-\sqrt{\lambda_j^{Q_N}}\sum_{i=1}^N v_j^{Q_N}(i) \mathds{1}_{P_i}\Big\|^2t+\sum_{j=d+1}^{d_N}\Big\|\sqrt{\lambda_j^{Q_N}}\sum_{i=1}^N v_j^{Q_N}(i) \mathds{1}_{P_i}\Big\|^2t,
\end{align*}
by Assumption \ref{assumption2}, we obtain
\begin{equation}\label{23}
  \sup_{t\in[0,T]}E\|W_t^{[N]}-W_t^Q\|^2\rightarrow 0,\quad as\ N\rightarrow\infty.
\end{equation}
Moreover,
\begin{align}\label{23a}
  E\int_0^T\|(\mathds{A}_s-\mathds{A}_s^{[N]})\tilde{y}_s\|^2ds\leq \frac{8B^4}{R^2} E\int_0^T \|(\Pi_s-\Pi_s^{[N]})\tilde{y}_s\|^2ds+2b^2\|M-M^{[N]}\|_{op}^2 E\int_0^T \|\tilde{y}_s\|^2ds.
\end{align}
By \cite[Th.9]{Gao1}, we can obtain $\Pi^{[N]}\rightarrow\Pi$ in $C_\mathcal{S}([0,T];\Sigma(L^2[0,1]))$, i.e.,
\begin{equation}\label{stronglyconverge}
  \lim_{N\rightarrow\infty}\sup_{t\in[0,T]}\|\Pi_t^{[N]}v-\Pi_tv\|=0,\quad \forall v\in L^2[0,1].
\end{equation}
Since $\|(\Pi_s^{[N]}-\Pi_s)\tilde{y}_s\|\leq \sup_{t\in[0,T]}\|(\Pi_t^{[N]}-\Pi_t)\tilde{y}_s\|$ holds for all $s$ and $\omega$, it follows from (\ref{stronglyconverge}) that, for all $s$ and $\omega$,
\begin{equation}\label{stronglyconverge1}
  \lim_{N\rightarrow\infty}\|(\Pi_s^{[N]}-\Pi_s)\tilde{y}_s\|\leq \lim_{N\rightarrow\infty}\sup_{t\in[0,T]}\|(\Pi_t^{[N]}-\Pi_t)\tilde{y}_s\|=0.
\end{equation}
By Lemma \ref{lemmabounded}, there exists a constant $C$ such that, for all $s$ and $\omega$,
\begin{equation}\label{dominated1}
  \|(\Pi_s^{[N]}-\Pi_s)\tilde{y}_s\|^2\leq \|\Pi_s^{[N]}-\Pi_s\|_{op}^2\|\tilde{y}_s\|^2\leq C\|\tilde{y}_s\|^2.
\end{equation}
This, combined with (\ref{stronglyconverge1}) and the Dominated Convergence Theorem, yields that
\begin{equation}\label{23b}
  \lim_{N\rightarrow\infty}E\int_0^T \|(\Pi_s^{[N]}-\Pi_s)\tilde{y}_s\|^2ds=0.
\end{equation}
From (\ref{23a}), (\ref{23b}) and Assumption \ref{assumption1}, we have
\begin{equation}\label{25}
  \lim_{N\rightarrow\infty}E\int_0^T\|(\mathds{A}_s-\mathds{A}_s^{[N]})\tilde{y}_s\|^2ds=0.
\end{equation}
By (\ref{22}), (\ref{23}), (\ref{25}), Assumption \ref{assumption1} and Lemma \ref{lemmabounded}, one has
\begin{equation}\label{26}
  \lim_{N\rightarrow\infty}\sup_{t\in [0,T]}E\|\tilde{y}_t-\bar{x}_t^{[N]}\|^2=0.
\end{equation}
Then we obtain the first conclusion.

Next we show that the second conclusion holds. In fact, it follows from \cite[Th.7.2]{Prato} that
\begin{equation}\label{24}
  \sup_{t\in[0,T]}E\|\tilde{y}_t\|^2\leq C(1+\|x_0\|^2).
\end{equation}
Since
\begin{equation}\label{26a}
  \sup_{t\in[0,T]}E\|\bar{x}_t^{[N]}\|^2\leq 2\sup_{t\in[0,T]}E\|\bar{x}_t^{[N]}-\tilde{y}_t\|^2+2\sup_{t\in[0,T]}E\|\tilde{y}_t\|^2,
\end{equation}
combining (\ref{26}), (\ref{24}) and (\ref{26a}), there is a constant $C$ such that
$\sup\limits_{1\leq N<\infty} \sup\limits_{t\in[0,T]}E\|\bar{x}_t^{[N]}\|^2\leq C$.
\end{proof}

\begin{remark}
A similar result to Lemma \ref{lemma3} has been previously obtained in \cite[Th. 3.3]{Dunyak3}, but here we provide a new proof approach.
\end{remark}

\begin{lemma}\label{lemma4}
There is a constant $C$ such that
$$
\lim_{N\rightarrow\infty}\sup_{t\in [0,T]}E\|\hat{x}_t^{[N]}-\tilde{y}_t\|^2=0,\quad \sup\limits_{1\leq N<\infty} \sup\limits_{t\in[0,T]}E\|\hat{x}_t^{[N]}\|^2\leq C.
$$
\end{lemma}
\begin{proof}
From (\ref{auxiliarystate}) and (\ref{auxiliarycontrol}), one has for a.s. $\omega$,
\begin{equation}\label{tildeyt}
  \tilde{y}_t=x_0+\int_0^t \Big\{\Big[(A-\frac{2B^2}{R}\Pi_s^\bot)I+bM\Big]\tilde{y}_s-\frac{2B^2}{R}\sum_{l=1}^\infty (\bar{\Pi}_s^l-\Pi_s^\bot)\varphi_s^l f_l^M\Big\}ds+\sigma W_t^Q,\quad \forall t\in[0,T].
\end{equation}
In view of (\ref{y_t}) and (\ref{tildeyt}), we have for a.s. $\omega$,
\begin{align}\label{28}
 \hat{x}_t^{[N]} & -\tilde{y}_t = x_0^{[N]}-x_0+\int_0^t \Big\{(A-\frac{2B^2}{R}\Pi_s^\bot)I(\hat{x}_s^{[N]}-\tilde{y}_s)+b(M^{[N]}\hat{x}_s^{[N]}-M\tilde{y}_s)\nonumber\\
 & -\frac{2B^2}{R}\Big[N\sum_{i=1}^N\mathds{1}_{P_i}\Big\langle \mathds{1}_{P_i},\sum_{l=1}^\infty (\bar{\Pi}_s^l-\Pi_s^\bot)\varphi_s^lf_l^M\Big\rangle-\sum_{l=1}^\infty (\bar{\Pi}_s^l-\Pi_s^\bot)\varphi_s^lf_l^M\Big]\Big\}ds+\sigma (W_t^{[N]}-W_t^Q),\quad \forall t\in[0,T].
\end{align}
Then it holds that
\begin{align}\label{28a}
  \|\hat{x}_t^{[N]}-\tilde{y}_t\|^2\leq & 3\|x_0^{[N]}-x_0\|^2+9t\int_0^t \Big[\|(A-\frac{2B^2}{R}\Pi_s^\bot)I(\hat{x}_s^{[N]}-\tilde{y}_s)\|^2+b^2\|M^{[N]}\hat{x}_s^{[N]}-M\tilde{y}_s\|^2\Big.\nonumber\\
  & \Big.+\frac{4B^4}{R^2}\Big\|N\sum_{i=1}^N\mathds{1}_{P_i}\Big\langle \mathds{1}_{P_i},\sum_{l=1}^\infty (\bar{\Pi}_s^l-\Pi_s^\bot)\varphi_s^lf_l^M\Big\rangle-\sum_{l=1}^\infty (\bar{\Pi}_s^l-\Pi_s^\bot)\varphi_s^lf_l^M\Big\|^2\Big]ds\nonumber\\
  & +3\sigma^2\|W_t^{[N]}-W_t^Q\|^2.
\end{align}
Using the triangle inequality yields
\begin{align}\label{29}
  b^2\|M^{[N]}\hat{x}_s^{[N]}-M\tilde{y}_s\|^2 & =b^2\|M^{[N]}\hat{x}_s^{[N]}-M^{[N]}\tilde{y}_s+M^{[N]}\tilde{y}_s-M\tilde{y}_s\|^2\nonumber\\
  & \leq 2b^2\sup_{1\leq N<\infty}\|M^{[N]}\|_{op}^2 \|\hat{x}_s^{[N]}-\tilde{y}_s\|^2+2b^2\|M^{[N]}-M\|_{op}^2\|\tilde{y}_s\|^2.
\end{align}
By (\ref{29}) and taking expectation on both sides of (\ref{28a}), one has
\begin{align*}
  E\|\hat{x}_t^{[N]}-\tilde{y}_t\|^2\leq & 3\|x_0^{[N]}-x_0\|^2+9T\Big(\sup_{t\in[0,T]}|A-\frac{2B^2}{R}\Pi_t^\bot|^2+2b^2\sup_{1\leq N<\infty}\|M^{[N]}\|_{op}^2\Big)\int_0^t E\|\hat{x}_s^{[N]}-\tilde{y}_s\|^2ds\\
  & +18b^2T\|M^{[N]}-M\|_{op}^2 E\int_0^T\|\tilde{y}_s\|^2ds\\
  & +\frac{36B^4}{R^2}TE\int_0^T \Big\|N\sum_{i=1}^N\mathds{1}_{P_i}\Big\langle \mathds{1}_{P_i},\sum_{l=1}^\infty (\bar{\Pi}_s^l-\Pi_s^\bot)\varphi_s^lf_l^M\Big\rangle-\sum_{l=1}^\infty (\bar{\Pi}_s^l-\Pi_s^\bot)\varphi_s^lf_l^M\Big\|^2ds\\
  & +3\sigma^2\sup_{t\in[0,T]}E\|W_t^{[N]}-W_t^Q\|^2.
\end{align*}
Since $E\|\hat{x}_t^{[N]}-\tilde{y}_t\|^2\leq 2\sup_{t\in[0,T]}E\|\hat{x}_t^{[N]}\|^2+2\sup_{t\in[0,T]}E\|\tilde{y}_t\|^2$, one has that $E\|\hat{x}_t^{[N]}-\tilde{y}_t\|^2$ is bounded in $[0,T]$. Applying Gronwall's inequality, we have
\begin{align}\label{30}
  \sup_{t\in[0,T]} & E\|\hat{x}_t^{[N]}-\tilde{y}_t\|^2\leq \Big[3\|x_0^{[N]}-x_0\|^2+18b^2T\|M^{[N]}-M\|_{op}^2 E\int_0^T\|\tilde{y}_s\|^2ds\Big.\nonumber\\
  & +\frac{36B^4}{R^2}TE\int_0^T \Big\|N\sum_{i=1}^N\mathds{1}_{P_i}\Big\langle \mathds{1}_{P_i},\sum_{l=1}^\infty (\bar{\Pi}_s^l-\Pi_s^\bot)\varphi_s^lf_l^M\Big\rangle-\sum_{l=1}^\infty (\bar{\Pi}_s^l-\Pi_s^\bot)\varphi_s^lf_l^M\Big\|^2ds\nonumber\\
  & +3\sigma^2\sup_{t\in[0,T]}E\|W_t^{[N]}-W_t^Q\|^2\Big]\exp\Big[9T^2\Big(\sup_{t\in[0,T]}|A-\frac{2B^2}{R}\Pi_t^\bot|^2+2b^2\sup_{1\leq N<\infty}\|M^{[N]}\|_{op}^2\Big)\Big].
\end{align}
Since there exists a constant $C$ such that
\begin{align}\label{30a}
  & \Big\|N\sum_{i=1}^N\mathds{1}_{P_i}\Big\langle \mathds{1}_{P_i},\sum_{l=1}^\infty (\bar{\Pi}_s^l-\Pi_s^\bot)\varphi_s^lf_l^M\Big\rangle-\sum_{l=1}^\infty (\bar{\Pi}_s^l-\Pi_s^\bot)\varphi_s^lf_l^M\Big\|^2\nonumber\\
  \leq & 2\Big\|N\sum_{i=1}^N\mathds{1}_{P_i}\Big\langle \mathds{1}_{P_i},\sum_{l=1}^\infty (\bar{\Pi}_s^l-\Pi_s^\bot)\varphi_s^lf_l^M\Big\rangle\Big\|^2+2\Big\|\sum_{l=1}^\infty (\bar{\Pi}_s^l-\Pi_s^\bot)\varphi_s^lf_l^M\Big\|^2\nonumber\\
  \leq & 4\|\sum_{l=1}^\infty (\bar{\Pi}_s^l-\Pi_s^\bot)\varphi_s^lf_l^M\Big\|^2\nonumber\\
  \leq & C\|\tilde{y}_s\|^2,
\end{align}
by Lemma \ref{lemmaapproximate} and Dominated Convergence Theorem, we obtain
\begin{equation}\label{31}
\lim_{N\rightarrow\infty}E\int_0^T \Big\|N\sum_{i=1}^N\mathds{1}_{P_i}\Big\langle \mathds{1}_{P_i},\sum_{l=1}^\infty (\bar{\Pi}_s^l-\Pi_s^\bot)\varphi_s^lf_l^M\Big\rangle-\sum_{l=1}^\infty (\bar{\Pi}_s^l-\Pi_s^\bot)\varphi_s^lf_l^M\Big\|^2ds=0.
\end{equation}
Combining (\ref{23}), (\ref{30}), (\ref{31}) and Assumption \ref{assumption1}, we derive
$\lim_{N\rightarrow\infty}\sup_{t\in[0,T]}E\|\hat{x}_t^{[N]}-\tilde{y}_t\|^2=0$.
Then we get the first conclusion. Similar to the proof of the second part of Lemma \ref{lemma3}, there is a constant $C$ such that $\sup\limits_{1\leq N<\infty} \sup\limits_{t\in[0,T]}E\|\hat{x}_t^{[N]}\|^2\leq C$.
\end{proof}

\begin{lemma}\label{lemma5}
There exists a constant $C$ such that
$$
\lim_{N\rightarrow\infty}\int_0^T E\|\bar{u}_t^{[N]}-\tilde{u}_t\|^2dt=0,\quad \sup\limits_{1\leq N<\infty} \sup\limits_{t\in[0,T]}E\|\bar{u}_t^{[N]}\|^2\leq C.
$$
\end{lemma}
\begin{proof}
From (\ref{centralizedcontrolrewrite}) and (\ref{tildeu}), applying the triangle inequality yields
\begin{align}\label{31a}
  \int_0^T E\|\bar{u}_t^{[N]}-\tilde{u}_t\|^2dt & =\frac{4B^2}{R^2}\int_0^T E\|\Pi_t^{[N]}\bar{x}_t^{[N]}-\Pi_t\tilde{y}_t\|^2dt\nonumber\\
  & =\frac{4B^2}{R^2}\int_0^T E\|\Pi_t^{[N]}\bar{x}_t^{[N]}-\Pi_t^{[N]}\tilde{y}_t+\Pi_t^{[N]}\tilde{y}_t-\Pi_t\tilde{y}_t\|^2dt\nonumber\\
  & \leq \frac{8B^2}{R^2}\int_0^T E\|\Pi_t^{[N]}(\bar{x}_t^{[N]}-\tilde{y}_t)\|^2dt+\frac{8B^2}{R^2}\int_0^T E\|(\Pi_t^{[N]}-\Pi_t)\tilde{y}_t\|^2dt.
\end{align}
Since
\begin{equation*}
  \|(\Pi_t^{[N]}-\Pi_t)\tilde{y}_t\|^2\leq 2\Big(\sup\limits_{1\leq N<\infty} \sup\limits_{t\in[0,T]}\|\Pi_t^{[N]}\|_{op}^2+\sup\limits_{t\in[0,T]}\|\Pi_t\|_{op}^2\Big)\|\tilde{y}_t\|^2,
\end{equation*}
by (\ref{stronglyconverge1}) and Dominated Convergence Theorem, one has
\begin{equation}\label{31b}
  \lim_{N\rightarrow\infty}\int_0^T E\|(\Pi_t^{[N]}-\Pi_t)\tilde{y}_t\|^2dt=0.
\end{equation}
By Lemmas \ref{lemmabounded} and \ref{lemma3},
\begin{equation}\label{31c}
  \int_0^T E\|\Pi_t^{[N]}(\bar{x}_t^{[N]}-\tilde{y}_t)\|^2dt\leq  \Big(\sup\limits_{1\leq N<\infty} \sup\limits_{t\in[0,T]}\|\Pi_t^{[N]}\|_{op}^2\Big)\int_0^T E\|\bar{x}_t^{[N]}-\tilde{y}_t\|^2dt\rightarrow 0.
\end{equation}
Then, by (\ref{31a}), (\ref{31b}), (\ref{31c}), it holds that
$\lim_{N\rightarrow\infty}\int_0^T E\|\bar{u}_t^{[N]}-\tilde{u}_t\|^2dt=0$.
Then we obtain the first result. Since
\begin{align*}
  \sup\limits_{t\in[0,T]}E\|\bar{u}_t^{[N]}\|^2=\frac{4B^2}{R^2}\sup\limits_{t\in[0,T]}E\|\Pi_t^{[N]}\bar{x}_t^{[N]}\|^2\leq \frac{4B^2}{R^2}\Big(\sup\limits_{1\leq N<\infty} \sup\limits_{t\in[0,T]}\|\Pi_t^{[N]}\|_{op}^2\Big) \Big(\sup\limits_{1\leq N<\infty} \sup\limits_{t\in[0,T]}E\|\bar{x}_t^{[N]}\|^2\Big),
\end{align*}
by Lemmas \ref{lemmabounded} and \ref{lemma3}, there is a constant $C$ such that
$\sup\limits_{1\leq N<\infty} \sup\limits_{t\in[0,T]}E\|\bar{u}_t^{[N]}\|^2\leq C$.
\end{proof}

\begin{lemma}\label{lemma6}
There exists a constant $C$ such that
$$
\lim_{N\rightarrow\infty}\int_0^T E\|\hat{u}_t^{[N]}-\tilde{u}_t\|^2dt=0,\quad \sup\limits_{1\leq N<\infty} \sup\limits_{t\in[0,T]}E\|\hat{u}_t^{[N]}\|^2\leq C.
$$
\end{lemma}

\begin{proof}
From (\ref{auxiliarycontrol}) and (\ref{hatuN}), one has
\begin{align*}
  \int_0^TE\|\hat{u}_t^{[N]}-\tilde{u}_t\|^2dt= & \int_0^TE \Big\|-\frac{2B}{R}\Pi_t^\bot I(\hat{x}_t^{[N]}-\tilde{y}_t)\Big.\\
  & \Big.-\frac{2B}{R}\Big[N\sum_{i=1}^N\mathds{1}_{P_i}\Big\langle \mathds{1}_{P_i},\sum_{l=1}^\infty (\bar{\Pi}_t^l-\Pi_t^\bot)\varphi_t^lf_l^M\Big\rangle-\sum_{l=1}^\infty (\bar{\Pi}_t^l-\Pi_t^\bot)\varphi_t^lf_l^M\Big]\Big\|^2dt\\
  \leq & \frac{8B^2}{R^2}\sup_{t\in [0,T]}|\Pi_t^\bot|^2 \int_0^TE \|\hat{x}_t^{[N]}-\tilde{y}_t\|^2dt\\
  & +\frac{8B^2}{R^2}\int_0^TE\Big\|N\sum_{i=1}^N\mathds{1}_{P_i}\Big\langle \mathds{1}_{P_i},\sum_{l=1}^\infty (\bar{\Pi}_t^l-\Pi_t^\bot)\varphi_t^lf_l^M\Big\rangle-\sum_{l=1}^\infty (\bar{\Pi}_t^l-\Pi_t^\bot)\varphi_t^lf_l^M\Big\|^2dt.
\end{align*}
Then ,by Lemma \ref{lemma4} and (\ref{31}), we have
$\lim_{N\rightarrow\infty}\int_0^T E\|\hat{u}_t^{[N]}-\tilde{u}_t\|^2dt=0$.
Then we obtain the first conclusion. It follows from \eqref{30a} that
\begin{align*}
  \sup\limits_{t\in[0,T]}E\|\hat{u}_t^{[N]}\|^2 & \leq \frac{8B^2}{R^2}\sup\limits_{t\in[0,T]}E\|\Pi_t^\bot I\hat{x}_t^{[N]}\|^2+\frac{8B^2}{R^2}\sup\limits_{t\in[0,T]}E\Big\|N\sum_{i=1}^N\mathds{1}_{P_i}\Big\langle \mathds{1}_{P_i},\sum_{l=1}^\infty (\bar{\Pi}_t^l-\Pi_t^\bot)\varphi_t^lf_l^M\Big\rangle\Big\|^2\\
  & \leq \frac{8B^2}{R^2}\sup\limits_{t\in[0,T]}E\|\Pi_t^\bot I\hat{x}_t^{[N]}\|^2+\frac{8B^2}{R^2}C\sup\limits_{t\in[0,T]}E\|\tilde{y}_t\|^2.
\end{align*}
Therefore, in view of (\ref{24}) and Lemma \ref{lemma4}, we can derive the other conclusion.
\end{proof}

Next, we propose the proof of Theorem \ref{theoremasymptotic}.
\begin{proof}[Proof of Theorem \ref{theoremasymptotic}]
From (\ref{J(u[N])=1/NJsoc}), (\ref{centralizedcontrolrewrite}), (\ref{centralizedcontrolith}), (\ref{J(hatu[N])=1/NJsoc}) and (\ref{hatuN}), one has
\begin{align}\label{21}
  \Big|\frac{1}{N}J_{soc}^N(\mathbf{\hat{u}})-\frac{1}{N}\inf_{\mathbf{u}\in\mathcal{U}_c}J_{soc}^N(\mathbf{u})\Big| & =|J(\hat{u}^{[N]})-J(\bar{u}^{[N]})|\nonumber\\
  & \leq\frac{1}{2}E\int_0^TQ\Big|\|(I-\Gamma M^{[N]})\hat{x}_t^{[N]}\|^2-\|(I-\Gamma M^{[N]})\bar{x}_t^{[N]}\|^2\Big|\nonumber\\
  &\quad +R\Big|\|\hat{u}_t^{[N]}\|^2-\|\bar{u}_t^{[N]}\|^2\Big|dt+\frac{1}{2}EQ_T\Big|\|\hat{x}_T^{[N]}\|^2-\|\bar{x}_T^{[N]}\|^2\Big|.
\end{align}
Applying Lemmas \ref{lemma3} and \ref{lemma4}, we have
\begin{align*}
  & \frac{1}{2}E\int_0^TQ\Big|\|(I-\Gamma M^{[N]})\hat{x}_t^{[N]}\|^2-\|(I-\Gamma M^{[N]})\bar{x}_t^{[N]}\|^2\Big|dt\\
  = & \frac{1}{2}E\int_0^TQ \Big[\|(I-\Gamma M^{[N]})\hat{x}_t^{[N]}\|+\|(I-\Gamma M^{[N]})\bar{x}_t^{[N]}\|\Big] \cdot\Big|\|(I-\Gamma M^{[N]})\hat{x}_t^{[N]}\|-\|(I-\Gamma M^{[N]})\bar{x}_t^{[N]}\|\Big|dt\\
  \leq & \frac{1}{2}E\int_0^TQ \|I-\Gamma M^{[N]}\|_{op} \cdot\big(\|\hat{x}_t^{[N]}\|+\|\bar{x}_t^{[N]}\|\big)\cdot\|(I-\Gamma M^{[N]})(\hat{x}_t^{[N]}-\bar{x}_t^{[N]})\|dt\\
  \leq & \frac{1}{2}Q \|I-\Gamma M^{[N]}\|_{op}^2 E\int_0^T \big(\|\hat{x}_t^{[N]}\|+\|\bar{x}_t^{[N]}\|\big)\cdot\|\hat{x}_t^{[N]}-\bar{x}_t^{[N]}\|dt\\
  \leq & \frac{1}{2}Q \|I-\Gamma M^{[N]}\|_{op}^2 \int_0^T \big[E\big(\|\hat{x}_t^{[N]}\|+\|\bar{x}_t^{[N]}\|\big)^2\big]^\frac{1}{2} \big(E\|\hat{x}_t^{[N]}-\bar{x}_t^{[N]}\|^2\big)^\frac{1}{2}dt\\
  \leq & \frac{1}{2}Q \|I-\Gamma M^{[N]}\|_{op}^2 \big(\sup_{t\in[0,T]}E\|\hat{x}_t^{[N]}-\bar{x}_t^{[N]}\|^2\big)^\frac{1}{2} \int_0^T \big[E\big(\|\hat{x}_t^{[N]}\|+\|\bar{x}_t^{[N]}\|\big)^2\big]^\frac{1}{2}dt\\
  \leq & \frac{\sqrt{T}}{2}Q \|I-\Gamma M^{[N]}\|_{op}^2 \big(\sup_{t\in[0,T]}E\|\hat{x}_t^{[N]}-\bar{x}_t^{[N]}\|^2\big)^\frac{1}{2} \Big[2\int_0^T E\|\hat{x}_t^{[N]}\|^2dt+2\int_0^T E\|\bar{x}_t^{[N]}\|^2dt\Big]^\frac{1}{2}\\
  \leq & \sqrt{T}Q \|I-\Gamma M^{[N]}\|_{op}^2 \Big(\sup_{t\in[0,T]}E\|\hat{x}_t^{[N]}-\tilde{y}_t\|^2+\sup_{t\in[0,T]}E\|\tilde{y}_t-\bar{x}_t^{[N]}\|^2\Big)^\frac{1}{2} \Big[\int_0^T E\|\hat{x}_t^{[N]}\|^2dt+\int_0^T E\|\bar{x}_t^{[N]}\|^2dt\Big]^\frac{1}{2}\\
  & \rightarrow 0,\ as\ N\rightarrow\infty.
\end{align*}
Similarly, we obtain
\begin{align*}
  & \frac{1}{2}EQ_T\Big|\|\hat{x}_T^{[N]}\|^2-\|\bar{x}_T^{[N]}\|^2\Big|\\
  = & \frac{1}{2}Q_T E\Big[\big(\|\hat{x}_T^{[N]}\|+\|\bar{x}_T^{[N]}\|\big)\big|\|\hat{x}_T^{[N]}\|-\|\bar{x}_T^{[N]}\|\big|\Big]\\
  \leq & \frac{1}{2}Q_T \big(2E\|\hat{x}_T^{[N]}\|^2+2E\|\bar{x}_T^{[N]}\|^2\big)^\frac{1}{2} \big(E\|\hat{x}_T^{[N]}-\bar{x}_T^{[N]}\|^2\big)^\frac{1}{2}\\
  \leq & \frac{1}{2}Q_T \Big(2\sup_{t\in[0,T]}E\|\hat{x}_t^{[N]}\|^2+2\sup_{t\in[0,T]}E\|\bar{x}_t^{[N]}\|^2\Big)^\frac{1}{2} \Big(\sup_{t\in[0,T]}E\|\hat{x}_t^{[N]}-\bar{x}_t^{[N]}\|^2\Big)^\frac{1}{2}\\
  & \rightarrow 0,\ as\ N\rightarrow\infty.
\end{align*}
In addition, it follows from Lemmas \ref{lemma5} and \ref{lemma6} that
\begin{align*}
  & \frac{1}{2}E\int_0^T R\Big|\|\hat{u}_t^{[N]}\|^2-\|\bar{u}_t^{[N]}\|^2\Big|dt\\
  = & \frac{R}{2}E\int_0^T \Big(\|\hat{u}_t^{[N]}\|+\|\bar{u}_t^{[N]}\|\Big)\Big|\|\hat{u}_t^{[N]}\|-\|\bar{u}_t^{[N]}\|\Big|dt\\
  \leq & \frac{R}{2}E\int_0^T (\|\hat{u}_t^{[N]}\|+\|\bar{u}_t^{[N]}\|) \|\hat{u}_t^{[N]}-\bar{u}_t^{[N]}\|dt\\
  \leq & \frac{R}{2}\int_0^T \Big[E(\|\hat{u}_t^{[N]}\|+\|\bar{u}_t^{[N]}\|)^2\Big]^\frac{1}{2} \Big(E\|\hat{u}_t^{[N]}-\bar{u}_t^{[N]}\|^2\Big)^\frac{1}{2}dt\\
  \leq & \frac{R}{2} \Big[\sup_{t\in[0,T]}E(\|\hat{u}_t^{[N]}\|+\|\bar{u}_t^{[N]}\|)^2\Big]^\frac{1}{2} \int_0^T \Big(E\|\hat{u}_t^{[N]}-\bar{u}_t^{[N]}\|^2\Big)^\frac{1}{2}dt\\
  \leq & \frac{R}{2}\sqrt{T} \Big[\sup_{t\in[0,T]}E(\|\hat{u}_t^{[N]}\|+\|\bar{u}_t^{[N]}\|)^2\Big]^\frac{1}{2} \Big(\int_0^T E\|\hat{u}_t^{[N]}-\bar{u}_t^{[N]}\|^2dt\Big)^\frac{1}{2}\\
  \leq & R\sqrt{T} \Big(\sup_{t\in[0,T]}E\|\hat{u}_t^{[N]}\|^2+\sup_{t\in[0,T]}E\|\bar{u}_t^{[N]}\|^2\Big)^\frac{1}{2} \Big(\int_0^T E\|\hat{u}_t^{[N]}-\tilde{u}_t\|^2dt+\int_0^T E\|\tilde{u}_t-\bar{u}_t^{[N]}\|^2dt\Big)^\frac{1}{2}\\
  & \rightarrow 0,\ as\ N\rightarrow\infty.
\end{align*}
Therefore, it holds that $  \Big|\frac{1}{N}J_{soc}^N(\mathbf{\hat{u}})-\frac{1}{N}\inf_{\mathbf{u}\in\mathcal{U}_c}J_{soc}^N(\mathbf{u})\Big|=o(1)$.
\end{proof}

Finally, we provide some examples of correlation matrices $Q_N$ and trace-class operators $Q$ satisfying Assumption \ref{assumption2}.
\begin{example}\label{example1}
Consider an $N\times N$ correlation matrix $Q_N=\big[\cos(\pi\frac{i-j}{N})\big]_{ij}$. Let
\begin{equation*}
v_1^{Q_N}=\sqrt{\frac{2}{N}}\Big(\cos\frac{\pi}{N},\cos\frac{2\pi}{N},\cdots,\cos\pi\Big)^T,\qquad v_2^{Q_N}=\sqrt{\frac{2}{N}}\Big(\sin\frac{\pi}{N},\sin\frac{2\pi}{N},\cdots,\sin\pi\Big)^T.
\end{equation*}
Then, one has $(v_1^{Q_N})^T v_1^{Q_N}+(v_2^{Q_N})^T v_2^{Q_N}=2$ and $(v_1^{Q_N})^T v_1^{Q_N}-(v_2^{Q_N})^T v_2^{Q_N}=\frac{2}{N}\sum_{i=1}^N\cos\frac{2\pi i}{N}=0$, which yields $\|v_1^{Q_N}\|=\|v_2^{Q_N}\|=1$. Since $
(v_1^{Q_N})^T v_2^{Q_N}=\frac{1}{N}\sum_{i=1}^N\sin\frac{2\pi i}{N}=0,$
we have that $v_1^{Q_N}$ and $v_2^{Q_N}$ are orthonormal. Since
$
Q_N=\frac{N}{2}\Big[v_1^{Q_N} (v_1^{Q_N})^T+v_2^{Q_N} (v_2^{Q_N})^T\Big],
$
it holds that
$Q_N v_1^{Q_N}=\frac{N}{2}v_1^{Q_N}$ and $ Q_N v_2^{Q_N}=\frac{N}{2}v_2^{Q_N}.$
Then we have $rank\ Q_N=2$, and eigenvalues and eigenvectors of $Q_N$ are $\frac{N}{2}$, $\frac{N}{2}$ and $v_1^{Q_N}$, $v_2^{Q_N}$, respectively. In addition, considering $Q(x,y)=\cos(\pi(x-y))$, $x,y\in[0,1]$, we define the integral operator as below
$$
(Qf)(x)=\int_0^1Q(x,y)f(y)dy,\quad \forall f\in L^2[0,1].
$$
The eigenvalues and eigenfunctions of $Q$ are $\lambda_1^Q=\lambda_2^Q=\frac{1}{2}$, $f_1^Q=\sqrt{2}\cos\pi x$, $f_2^Q=\sqrt{2}\sin\pi x$, which yields that $Q$ is trace-class. Thus
\begin{align}\label{unknown1}
  \Big\|\sqrt{\lambda_1^Q}f_1^Q-\sqrt{\lambda_1^{Q_N}}\sum_{i=1}^Nv_1^{Q_N}(i)\mathds{1}_{P_i}\Big\|_2^2
  =&\Big\|\cos\pi(\cdot)-\sqrt{\frac{N}{2}}\sum_{i=1}^N\mathds{1}_{P_i}(\cdot)\sqrt{\frac{2}{N}}\cos\frac{i\pi}{N}\Big\|_2^2\nonumber\\
  =&\int_0^1\Big|\cos\pi x-\sum_{i=1}^N\mathds{1}_{P_i}(x)\cos\frac{i\pi}{N}\Big|^2dx\nonumber\\
  =&\sum_{i=1}^N\int_{P_i}\big(\cos\pi x-\cos\frac{i\pi}{N}\big)^2dx\nonumber\\
  =&\frac{1}{N}\sum_{i=1}^N \big(\cos\pi \xi_i-\cos\frac{i\pi}{N}\big)^2,\quad \xi_i\in [\frac{i-1}{N},\frac{i}{N}]\nonumber\\
  =&\frac{4}{N}\sum_{i=1}^N\sin^2\frac{\pi}{2}(\xi_i+\frac{i}{N})\sin^2\frac{\pi}{2}(\xi_i-\frac{i}{N})\nonumber\\
  \leq &\frac{4}{N}\sum_{i=1}^N\Big|\frac{\pi}{2}(\xi_i-\frac{i}{N})\Big|^2\leq \frac{\pi^2}{N^2}\rightarrow 0.
\end{align}
Similarly, one has
\begin{equation}\label{unknown2}
  \Big\|\sqrt{\lambda_2^Q}f_2^Q-\sqrt{\lambda_2^{Q_N}}\sum_{i=1}^Nv_2^{Q_N}(i)\mathds{1}_{P_i}\Big\|_2^2\rightarrow 0.
\end{equation}
By (\ref{unknown1}) and (\ref{unknown2}), $Q_N$ and $Q$ considered in this example satisfy Assumption \ref{assumption2}.
\end{example}

\begin{example}
Consider an $N\times N$ double-constant matrix $Q_N$, where the $(i,j)$-th entry is given by:
\begin{equation*}
  {\rho}_{ij}=\left\{
  \begin{split}
    &1,\ & if\ i=j,\\
    &1-\frac{1}{2N},\ & if\ i\neq j.
  \end{split}
  \right.
\end{equation*}
According to \cite{Neill}, we obtain that $rank\ Q_N=N$, the corresponding eigenvalues are $\lambda_1^{Q_N}=1+(N-1)(1-\frac{1}{2N})$, $\lambda_2^{Q_N}=\cdots=\lambda_N^{Q_N}=\frac{1}{2N}$, and $v_1^{Q_N}=\frac{1}{\sqrt{N}}(1,1,\cdots,1)^T$ is an eigenvector with respect to $\lambda_1^{Q_N}$.
Considering $Q(x,y)=1$, the corresponding trace-class operator is defined as the one in Example \ref{example1} with eigenvalue $\lambda^Q=1$ and eigenfunction $f^Q=\mathds{1}(x)$. Thus, one has
\begin{equation}\label{unknown3}
  \Big\|\sqrt{\lambda^Q}f^Q-\sqrt{\lambda_1^{Q_N}}\sum_{i=1}^N v_1^{Q_N}(i) \mathds{1}_{P_i}\Big\|_2^2=\left|1-\sqrt{\frac{1}{N}+\frac{N-1}{N}(1-\frac{1}{2N})}\right|^2 \|\mathds{1}\|_2^2\rightarrow 0,
\end{equation}
and
\begin{equation}\label{unknown4}
  \frac{1}{N}\sum_{j=2}^N\lambda_j^{Q_N}=\frac{N-1}{2N^2}\rightarrow 0.
\end{equation}
By (\ref{unknown3}) and (\ref{unknown4}), $Q_N$ and $Q$ considered in this example satisfy Assumption \ref{assumption2}.

\end{example}

\section{Finite rank graphon case}

\qquad Due to the presence of infinite sum, calculating the decentralized strategies \eqref{constructedcontrol} directly is extremely difficult. In this section, we investigate the finite rank graphon case that makes the calculation of decentralized strategies more tractable. Consider the following finite rank assumption which is also appeared in \cite{Parise,Gao,Dunyak1,Dunyak2,Tchuendom}.

\begin{assumption}\label{finiteeigenvalues}
The graphon operator $M$ corresponding to $M(\cdot,\cdot)$ defined by (\ref{graphonoperator}) admits finite non-zero eigenvalues $\{\lambda_l^M\}_{l=1}^L$ with a set of orthonormal eigenfunctions $\{f_l^M\}_{l=1}^L$.
\end{assumption}

Under the above assumption, we can rewrite the decentralized strategies \eqref{constructedcontrol} as follows
\begin{equation}\label{constructedcontrol2}
  \hat{u}_t^i=-\frac{2B}{R}\Pi_t^\bot\hat{x}_t^i-\frac{2B}{R}N \Big\langle \mathds{1}_{P_i},\sum_{l=1}^L (\bar{\Pi}_t^l-\Pi_t^\bot)\varphi_t^lf_l^M\Big\rangle
\end{equation}
with $\varphi_t^l=\langle \tilde{y}_t,f_l^M\rangle$. However, it is not without its limitations, as \eqref{constructedcontrol2} relies on a continuous modification of the solution of the abstract stochastic evolution equation \eqref{auxiliarystate}. We wish to further reduce the complexity of calculation. It follows from \eqref{auxiliarystate} that for a.s. $\omega$
\begin{equation*}
 \varphi_t^l= \langle \tilde{y}_t,f_l^M\rangle=\langle x_0,f_l^M\rangle+\int_0^t \langle (AI-\frac{2B^2}{R}\Pi(s)+bM)\tilde{y}_s,f_l^M\rangle ds+\sigma\langle W_t^Q,f_l^M\rangle,\quad t\in[0,T].
\end{equation*}
Since
$$
\langle \Pi_t\tilde{y}_t,f_l^M\rangle=\Big\langle \Pi_t^\bot \tilde{y}_t+\sum_{k=1}^L (\bar{\Pi}_t^k-\Pi_t^\bot)\langle f_k^M,\tilde{y}_t\rangle f_k^M,f_l^M\Big\rangle=\bar{\Pi}_t^l\langle \tilde{y}_t,f_l^M\rangle,
$$
$$
\langle M\tilde{y}_t,f_l^M\rangle=\Big\langle \sum_{k=1}^L \lambda_k^M\langle \tilde{y}_t,f_k^M\rangle f_k^M,f_l^M\Big\rangle=\lambda_l^M\langle \tilde{y}_t,f_l^M\rangle,
$$
and
$$
\langle W_t^Q,f_l^M\rangle=\Big\langle \sum_{j=1}^d \sqrt{\lambda_j^Q}W_t^jf_j^Q,f_l^M\Big\rangle=\sum_{j=1}^d \sqrt{\lambda_j^Q}\langle f_j^Q,f_l^M\rangle W_t^j,
$$
we have that $\varphi^l$ satisfies for a.s.
\begin{equation}\label{varphil}
  \varphi_t^l=\langle x_0,f_l^M\rangle+\int_0^t (A-\frac{2B^2}{R}\bar{\Pi}_s^l+b\lambda_l^M)\varphi_s^l ds+\sigma \sum_{j=1}^d \sqrt{\lambda_j^Q}\langle f_j^Q,f_l^M\rangle W_t^j, \quad t\in[0,T],
\end{equation}
which yields that $\varphi^l$ is the unique solution of SDE \eqref{varphil}. Thus, we can propose the following decentralized strategies for $i$th agent:
\begin{equation}\label{constructedcontrol3}
  \hat{u}_t^i=-\frac{2B}{R}\Pi_t^\bot\hat{x}_t^i-\frac{2B}{R}N \Big\langle \mathds{1}_{P_i},\sum_{l=1}^L (\bar{\Pi}_t^l-\Pi_t^\bot)\breve{\varphi}_t^lf_l^M\Big\rangle,
\end{equation}
where $\breve{\varphi}^l$ is any solution of SDE \eqref{varphil}, which satisfies $\mathbb{P}(\{\varphi_t^l=\breve{\varphi}_t^l,t\in[0,T]\})=1$. By a slight abuse of notation, we still use the notation $\hat{u}_t^i$ in \eqref{constructedcontrol3}. It is noteworthy that, compared to \eqref{constructedcontrol}, the strategies \eqref{constructedcontrol3} is relatively easier to calculate. In addition, it is obviously that \eqref{constructedcontrol3} has asymptotic social optimality.

\section{Conclusions and future work}
\qquad In this paper, we study social optimality for a class of large population linear quadratic systems on large-scale graphs. By transforming the social control problem into an infinite dimensional stochastic optimal control problem, we obtain the centralized feedback type centralized control. In addition, we design a set of decentralized strategies in a privacy-preserving pattern, and then show their asymptotical social optimality. It is noteworthy that, in this work, the "decentralized" refers specifically to privacy preservation. The analysis of traditional decentralized strategies based on our models is much more complicated and beyond the scope of our research, presenting an avenue worthy of further investigation.

\section*{Acknowledgement}
\qquad The authors would like to thank Prof. Qi L\"{u} and Dr. Fubo Li as well as Ms. Yue Zeng for their helpful suggestions and discussions.

\end{document}